\documentclass[final]{siamltex}

\usepackage{graphicx,pgf,amsfonts,amssymb,amsmath,eepic}
\usepackage{theorem}
\usepackage{hyperref}

\title{A Proximal Primal-Dual Approach to Generalized JKO Schemes for Doubly Nonlinear Parabolic Equations}

\author{%
  Luis M. Brice\~no-Arias\thanks{%
    Departamento de Matem\'atica, Universidad T\'ecnica Federico Santa Mar\'ia,
    Santiago, Chile
(\email{luis.briceno@usm.cl}).}%
  \and
  Jos\'e A.~Carrillo\thanks{%
    Mathematical Institute, University of Oxford,
    Oxford OX2 6GG, United Kingdom(\email{jose.carrillo@maths.ox.ac.uk}).}%
  \and
  Dante Kalise\thanks{%
    Department of Mathematics, Imperial College London,
    South Kensington Campus, London SW7 2AZ, United Kingdom
    (\email{dkaliseb@ic.ac.uk}).}%
  \and
  Francisco J.~Silva\thanks{%
XLIM, Universit\'e de Limoges, 87060 Limoges, France
(\email{francisco.silva@unilim.fr}).}%
  \and
  Li Wang\thanks{%
    School of Mathematics, University of Minnesota,
 Minneapolis, MN 55455, United States
(\email{liwang@umn.edu}).}%
}

\newcommand{\RP}{\ensuremath{\left[0,+\infty\right[}}

\newcommand{\RPP}{\ensuremath{\left]0,+\infty\right[}}

\newcommand{\RPX}{\ensuremath{\left[0,+\infty\right]}}
\newcommand{\RX}{\ensuremath{\left]-\infty,+\infty\right]}}

\newcommand{\inte}{\ensuremath{\text{\rm int\,}}}
\newcommand{\RR}{{\mathbb R}}
\newcommand{\R}{{\mathbb R}}

\newcommand{\NN}{{\mathbb N}}

\newcommand{\dd}{\;{\rm d}}

\newcommand{\PP}{\mathcal P}
\newcommand{\I}{\mathcal I}

\newcommand{\un}{{\rm 1\kern -2.5pt l}}
\newcommand{\id}{{\rm Id}}
\newcommand{\email}[1]{{\small E-mail: {\textsf {#1}}}}

\def\derpar#1#2{\frac{\partial#1}{\partial#2}}

\def\thefigure{{\arabic{section}.\arabic{figure}}}

\newcommand{\menge}[2]{\big\{{#1}~\big |~{#2}\big\}} 

\newcommand{\dom}{\ensuremath{\text{\rm dom}\,}}

\newcommand{\deltaxbf}{\Delta {\boldsymbol{x}}}
\newcommand{\Dtime}{1}
\newcommand{\prox}{\text{\rm Prox}}

\newcommand{\Amat}{\mathsf{A}}

\newcommand{\rd}{\mathrm{d}}

\newcommand{\C}{\mathcal{C}}

\theoremstyle{plain}{\theorembodyfont{\rmfamily}%
}
\theoremstyle{plain}{\theorembodyfont{\em\rmfamily}%
\newtheorem{prop}[theorem]{Proposition}}
\theoremstyle{plain}{\theorembodyfont{\em\rmfamily}%
\newtheorem{remark}[theorem]{Remark}}

\begin{document}

\maketitle

\begin{abstract}
Variational methods based on optimization strategies are proposed to numerically solve a large family of nonlinear partial differential equations. They are all particular instances of gradient flows with general costs, including the $p$-Laplace equation and flux-limited equations such as the relativistic heat equation. This is achieved by computing explicit formulas for proximal operators with general costs amenable to efficient numerical approximation. We showcase our numerical approach via validation of the results by recovering the qualitative behavior of particular known cases of this large family of steepest descent evolutions.
\end{abstract}



\section{Introduction}
This work develops a novel convex optimization scheme to solve a family of doubly nonlinear parabolic equations that arise as gradient flows with general cost functions. We study equations of the form 
\begin{equation} 
\label{contgen}
\derpar{\rho}{t} = \nabla\cdot \left(
\rho\nabla c^*\left[ \nabla  U' ( \rho ) \right]\right),
\end{equation}
where $c\colon\RR^{d}\to\RPX$ is an even convex cost function with $c(0)=0$, assumed differentiable in the interior of its domain, $c^*$ denotes its Fenchel conjugate, $\rho(t,\cdot)$ represents a time-dependent probability measure on $\R^d$ ($d\geq 1$), and $U\colon\left[0,+\infty\right[\to \R$ is a density of internal energy.

The velocity field in \eqref{contgen} has the form $-\nabla c^*\left[ \nabla \frac{\delta {\cal F}}{\delta \rho}\right]$, where $\frac{\delta {\cal F}}{\delta \rho}=U' \left ( \rho \right )$ and ${\cal F}(\rho) = \int_{\R^d} U(\rho)\dd x$ represents the free-energy functional. This functional exhibits monotone dissipation along solution trajectories according to
\begin{equation*}
\frac {\dd} {\dd t}{\cal F}(\rho(\cdot,t))= -\int_{\R^d}
\nabla c^*\left[ \nabla \frac{\delta {\cal F}}{\delta \rho}(\rho(x,t))\right] \cdot \nabla \frac{\delta {\cal F}}{\delta \rho}(\rho(x,t))\,\rho(x,t)\dd x\leq 0,
\end{equation*}
establishing the gradient flow structure for general costs as introduced in \cite{Brenier03}, and generalizing the classical Euclidean gradient flow theory \cite{AGS}.

Two prominent classes of equations motivate our study. First, the $p$-Laplacian and doubly nonlinear equations arising from power-law costs $c(x)={|x|^q}/{q}$, where $|\cdot|$ denotes the euclidean distance in $\RR^{d}$ and $q$ is conjugate to $1<p<\infty$, combined with appropriate energy densities $U(s)=s\log s$ or $U(s)={s^m}/{(m-1)}$ with $m>1$. These yield equations of the form
\begin{equation}\label{doubly}
\frac{\partial\rho}{\partial t} = \nabla\cdot \left(\left|\nabla \rho^m \right|^{p-2}\nabla \rho^m \right),
\end{equation}
which include the porous medium equation and $p$-Laplacian equations as special cases \cite{Ot96,Agu,Agu2,ABC}.

Second, flux-limited equations are another large family of equations included in \eqref{contgen}. They appear in relativistic flows, see
\cite{ACMM,ACMM2,MP,BP1,BP2}, and mathematical biology as flux-limited versions of the classical Keller-Segel model for chemotaxis, see \cite{BBNS12,PVW20}. 
A particular example of flux-limited equations is usually referred as the relativistic heat equation, as introduced in \cite{Brenier03},
\begin{equation}\label{relheat}
\derpar{\rho}{t} =
\nabla\cdot\left(\rho\frac{\nabla\rho}{\sqrt{\rho^2+|\nabla\rho|^2}}\right)=
\nabla\cdot\left(\rho\frac{\nabla\log\rho}{\sqrt{1+|\nabla\log\rho|^2}}\right).
\end{equation}
Here, the cost function is given by
\begin{equation}
\label{e:defc}
c(x)=
\left\{\begin{array}{cl}1-\sqrt{1-|x|^2} &\mbox{ if }|x|\leq 1,\\
+\infty &\mbox{ if } |x|> 1,\end{array}\right.
\end{equation}
for which $c^*(x)=\sqrt{1+|x|^2}-1$, and the energy density is $U(s)=s\log s$.

The two aforementioned particular costs will receive
special attention in our analysis due to their interest in applications to gradient flows,
however we will develop our approach in complete generality. 

Our computational approach builds upon the variational formulation of these gradient flows through a generalized Jordan-Kinderlehrer-Otto (JKO) scheme. The JKO scheme \cite{Ot96,JKO} was originally proposed for nonlinear parabolic PDEs with gradient flow structure and Euclidean cost $c(x)=|x|^2/2$, and where the discrete time evolution of the density is given by a proximal operator in the Wasserstein-2 metric
\begin{equation}\label{jkoprox}
 \rho^{n+1} \in \prox_{\Delta t {\cal F}}^{{\cal W}_2}(\rho^n) :=\underset{\rho}{\operatorname{argmin}} \left\lbrace
 \frac{1}{2\,\Delta t} {\cal W}_2^2(\rho^n_{\Delta t},\rho) + {\cal F}(\rho) 
 \right\rbrace.
 \;\,
\end{equation}
This scheme and its generalization will be presented in detail in the next section. However, it is worth mentioning that the eq. \eqref{jkoprox} has inspired a substantial amount of numerical approximation schemes for nonlinear parabolic PDEs of the form \eqref{contgen}:  the classical approach based on the inverse of the distribution function in one dimension \cite{CCM13}, the evolution of diffeomorphisms \cite{CM09,Junge2017}, and entropic regularization strategies \cite{Peyre2015,Carlier2017,MS20}. We refer to \cite{CMW21} for a complete survey on the subject. 

The aforementioned works circumvent a fundamental difficulty associated with the construction of schemes of the form \eqref{jkoprox}: the computation of the ${\cal W}_2$ distance. Following the pioneering work by Benamou and Brenier  \cite{BB}, an alternative is to cast the ${\cal W}_2$ distance between two densities as a PDE-constrained optimization problem over the continuity equation. While it is certainly more computationally expensive than resorting to entropic regularization techniques, the PDE-constrained optimization approach provides a convex variational framework without the need of regularization. This is particularly relevant for an implicit scheme built on \eqref{jkoprox}, which is prone to accumulate regularization error over long-term evolution. The idea of computing the ${\cal W}_2$ distance using a fluid flow formulation in conjunction with convex optimization methods has been explored in \cite{BB,PPO,BC15} in the context of optimal transport, in \cite{BC15,BAKS,briceno2019implementation} for mean field games, and in \cite{CCWW} in the framework of the JKO scheme with Euclidean cost and nonlinear mobilities \cite{CWW}. All these works can be understood as instances of the JKO scheme where, by following a discretize-then-optimize approach, after a suitable spatial discretization, every iteration of the scheme is cast as a large-scale convex optimization problem. This problem is solved using a proximal primal-dual splitting method \cite{CP11}, where the proximal operator calculation distillates in two fundamental operations: the solution of a large-scale linear system of equations arising from density discretization, and a scalar root-finding procedure in every grid-point, depending on both the free energy ${\cal F}$ and the cost $c(x)$.

In this work, our contribution focuses on the study of primal-dual proximal operator methods for the approximation of a generalized JKO scheme where the ${\cal W}_2$ distance in \eqref{jkoprox} is replaced by a generalized transport distance which covers PDEs of the form \eqref{contgen}. The core of our numerical approach is to derive explicit formulas for the proximal operators associated with the generalized cost functions under consideration, enabling efficient root finding via Newton-Raphson, and generating fast iterations of the primal-dual solver driving the JKO evolution.

The rest of the paper is structured as follows. In Section 2, we revisit the JKO scheme, its generalization, and the Benamou-Brenier formulation of the optimal transport problem. Section 3 is devoted to the discretization of the generalized JKO scheme and its solution using a proximal primal-dual splitting method. Section 4 is devoted to convex analysis methods for the calculation of proximal operators leading to transcendental one-dimensional nonlinear equations to be numerically solved. Finally, Section 5 presents numerical experiments for the $p-$Laplacian and the relativistic heat equation, assessing the overall performance of the proposed scheme to reproduce the essential features of the solutions regarding asymptotic and qualitative behavior.


\section{Generalised JKO schemes for general nonlinear degenerate parabolic equations} 

Let $\Omega\subset\RR^{d}$ be convex, open and bounded,  and let $\PP(\Omega)$ be the space of probability measures 
over $\Omega$.  We denote by $\PP^{{\rm ac}}(\Omega)$ the subset of $\PP(\Omega)$ formed
by absolutely continuous measures with respect to the Lebesgue measure. Given $\rho_0, \rho_1\in\PP(\Omega)$ and a measurable map $T\colon\Omega \to \Omega$, we say that $T$ 
{\it transports} $\rho_0$ onto $\rho_1$, and we note $\rho_1=T\# \rho_0$, if for any measurable set $B \subset\Omega$, $\rho_1(B)=\rho_0(T^{-1}(B))$, or
equivalently,
\begin{equation*}
\int_{\Omega} \zeta(T(x))\dd \rho_0(x) = \int_{\Omega} \zeta(y)
\dd \rho_1(y) \quad \forall \zeta \in \mathcal{C}_b(\Omega),
\end{equation*}
where $\mathcal{C}_b(\Omega)$ denotes the space of real-valued continuous and bounded functions defined on $\Omega$. The euclidean Wasserstein distance ${\cal W}_2(\rho_0,\rho_1)$ between $\rho_0$ and 
$\rho_1$ can be defined through
\begin{equation}
\label{eq:defMonge}
{\cal W}_2^2(\rho_0, \rho_1) :=  \inf_{T :\ \rho_1=T\# \rho_0} \int_{\Omega}
|x-T(x)|^2\dd\rho_0(x),
\end{equation}
which can be shown to coincide with the following linear programming relaxation, due to Kantorovich, 
\begin{equation}
\label{eq:defwasserstein}
{\cal W}_2^2(\rho_0, \rho_1) = \inf_{\Pi\in\Gamma} \left\{ \int_{\Omega \times
\Omega} \vert x - y \vert^2 \, \dd \Pi(x, y) \right\},
\end{equation}
where $\Pi$ runs over the set of transference plans $\Gamma$, 
that
is, the set of joint probability measures on $\Omega\times\Omega$
with marginals $\rho_0$ and $\rho_1$. Notice that, by the Kantorovich duality theorems \cite[Chapter 1]{Villani}, the infimum in~\eqref{eq:defwasserstein} is actually a minimum. If $\rho_0,\,\rho_1\in\PP^{{\rm ac}}(\Omega)$, an alternative expression, due to Benamou and Brenier~\cite{BB} (see also~\cite[Theorem 2.2]{Brenier03}), which is very useful in the numerical approximation of ${\cal W}_{2}(\rho_0,\rho_1)$, is given by
\begin{eqnarray}\label{BB4}
\frac12 {\cal W}_2^2(\rho_0,\rho_1)  =  \min_{(\rho, \boldsymbol{m}) \in \C_1}  \int_0^1 
\int_\Omega \Phi(\rho(x,t), \boldsymbol{m}(x,t))\,\rd x \rd t,
\end{eqnarray}
where $\Phi\colon\RR\times\RR^{d}\to\RR\cup\{+\infty\}$ is the perspective function of $|\cdot|^2/2$, see \cite{Rock66}, given by
\begin{equation} \label{Jdef}
\Phi(\rho, \boldsymbol{m}) = 
\begin{cases}
\dfrac{|\boldsymbol{m}|^2}{2 \rho} &\textrm{if } \rho>0, \\
0 & 
\textrm{if } (\rho, \boldsymbol{m})=(0,0), \\
+\infty & \textrm{otherwise\,,}
\end{cases}
\end{equation}
and $(\rho,\boldsymbol{m})\in \C_1$ if $\rho\colon[0,1]\to\PP^{{\rm ac}}(\Omega)$ and $\boldsymbol{m}\colon[0,1]\to(\PP^{{\rm ac}}(\Omega))^{d}$ are measurable, 
\begin{equation*}
\rho(\cdot,0) = \rho_0\quad\text{and}\quad\rho(\cdot, \Dtime) = \rho_{\Dtime} 
\text{ on } \Omega,
 \label{BB6}
\end{equation*}
and the following equation holds in the distributional sense 
\begin{align*}
\frac{\partial\rho}{\partial t} + \nabla \cdot \boldsymbol{m} = 0 &\quad\text{ on } \Omega \times[0,1],\\ 
 \boldsymbol{m} \cdot \nu = 0   &\quad\text{ on } \partial \Omega \times [0,1], 
\end{align*}
where 
$\nu\colon\partial\Omega\to\RR^d$ satisfies that, for every $x\in\partial\Omega$, $\nu(x)$ is a unit outward normal to $\Omega$ at $x$. Unlike~\eqref{eq:defMonge}, the new formulation (\ref{BB4}) is a convex minimization problem with linear constraints, and hence, after discretization, is amenable to be solved by convex optimization techniques, see \cite{BB,CP19} and the references therein. 

Associated with the Euclidean Wasserstein distance ${\cal W}_2$, a variational scheme for
the doubly nonlinear equation and the heat equation was 
introduced
in \cite{Ot96,JKO}, and consequently generalized to all the
equations of the form \eqref{contgen} in \cite{Agu,AGS}. This
variational scheme reads as:
\begin{equation}\label{eq:defminimisepme}
 \rho^{n+1}_{\Delta t} \in \underset{\rho \in \PP^{{\rm ac}}(\Omega)}{\operatorname{argmin}} \left\lbrace
 \frac{1}{2\,\Delta t} {\cal W}_2^2(\rho^n_{\Delta t},\rho) + {\cal F}(\rho) 
 \right\rbrace
 \;,
\end{equation}
for a fixed time step $\Delta t>0$ and an initial datum $\rho^0\in
\PP^{{\rm ac}}(\Omega)$. This steepest descent scheme can be understood \cite{AGS} as an implicit
time discretization of an abstract gradient flow equation in the
space of probability measures. We refer to
\cite{Otto01,Villani,AGS,Carrillo-McCann-Villani03,Carrillo-McCann-Villani06,Santambrogio2017} for a deeper
discussion, the heuristics and the rigorous sense of the gradient
flow structure. In \cite{AGS}, it is proved that suitable time
interpolation of the solutions of this variational scheme approaches
the solution of the limiting equation \eqref{contgen}, with quadratic cost $c(x)={|x|^2}/{2}$, at first
order in time in the ${\cal W}_2$ sense.

More general equations \eqref{contgen} can be treated by allowing
more general distances induced by different cost functionals for
the transport of mass from location $x\in\Omega$ to location $y\in\Omega$. More precisely, let $\phi\colon\RP\to\RPX$ be increasing, 
strictly convex and supercoercive, i.e., $\lim_{\xi\to+\infty}\phi(\xi)/\xi=+\infty$, define $c\colon\RR^{d}\to\RPX$ by $c=\phi\circ|\cdot|$, and assume that $c$ is differentiable in the interior of its domain. Then the generalized variational
scheme to equation \eqref{contgen} is
\begin{equation}\label{vergen}
\rho^{n+1}_{\Delta t} \in \underset{\rho \in\PP^{{\rm ac}}(\Omega)}{\operatorname{argmin}}
 \left\{ \Delta t {\cal W}_{c}^{\Delta t}(\rho^n_{\Delta t},\,\rho) + {\cal F}(\rho) 
 \right\},
\end{equation}
where ${\cal W}_{c}^{\Delta t}$ is defined as 
\begin{equation}\label{BB4c}
{\cal W}_{c}^{\Delta t}(\rho_0,\rho_1)=\min_{\Pi\in\Gamma} \left\{ \int_{\Omega \times
\Omega}c\left(\frac{x-y}{\Delta t}\right)\dd\Pi(x, y) \right\}\,.
\end{equation}
We refer to \cite{Agu} for details. It follows from~\eqref{vergen} and the extension of~\eqref{BB4} in~\cite{Brenier03}, which yields a dynamic reformulation of $ {\cal W}_{c}^{\Delta t}(\rho^n_{\Delta t},\,\rho)$, that $\rho^{n+1}_{\Delta t}$ can be determined as $\rho^{n+1}_{\Delta t}=\rho^*(\cdot,\Delta t)$ with $\rho^*$ obtained by minimizing the functional
\begin{eqnarray} 
\label{rhonp}
  \min_{(\rho, \boldsymbol{m}) \in \C_{\Delta t}^{n}}  \int_0^{\Delta t} 
\int_\Omega \Phi_c(\rho(x,t), \boldsymbol{m}(x,t))\rd x \rd t + {\cal F}(\rho(\cdot, 
\Delta t)),
\end{eqnarray}
where  $\Phi_{c} :\RR\times\RR^{d}\to\RPX$ is the perspective function of $c$, given by
\begin{align} \label{Jdefc}
\Phi_c(\rho, \boldsymbol{m}) &= 
\begin{cases}
\rho\, 
c\left(\frac{\boldsymbol{m}}{\rho}\right), \quad & \textrm{if } \rho>0;\\
0, & \textrm{if } (\rho, \boldsymbol{m})=(0,0);\\
+\infty, & \textrm{otherwise},
\end{cases}
\end{align}
and $(\rho,\boldsymbol{m})\in \C_{\Delta t}^{n}$ if $\rho\colon[0,\Delta t]\to\PP^{{\rm ac}}(\Omega)$ and $\boldsymbol{m}\colon[0,\Delta t]\to(\PP^{{\rm ac}}(\Omega))^{d}$ are measurable, $\rho(\cdot,0)=\rho^n_{\Delta t}$,
and the following equation holds in the distributional sense 
\begin{align}
\frac{\partial\rho}{\partial t} + \nabla \cdot \boldsymbol{m} = 0 &\quad\text{ on } \Omega \times[0,\Delta t],
 \label{CC1}\\ 
 \boldsymbol{m} \cdot \nu = 0   &\quad\text{ on } \partial \Omega \times [0,\Delta t]. 
 \label{CC2}
\end{align}
The variational scheme~\eqref{vergen} was proven to be convergent in a very general setting
\cite{Agu,AGS,MP,BP1,BP2,CGW} for the $p$-laplacian equations, the
doubly-nonlinear equations \eqref{doubly} and the relativistic heat equation \eqref{relheat2}.

As already mentioned, we will consider two important relevant choices of the cost function $c$. On the one hand, we will deal with the power-law cost $c(x)={|x|^q}/{q}$ with $1<q<\infty$ leading
to variational schemes approximating, for instance, the doubly nonlinear equation 
\begin{equation*}
\frac{\partial\rho}{\partial t} = \nabla\cdot \left[\left|\nabla \rho^m \right|^{p-2}\nabla \rho^m \right]\,,
\end{equation*}
with $p$ the conjugate exponent of $q$, $m>1$, and  the energy $U(s)=s^m/(m-1)$. On the other hand, we will introduce a generalized cost interpolating between the quadratic and the relativistic cost in \eqref{e:defc} given by
\begin{equation*}
\label{e:defc3}
c_{\alpha,k}(x)= \frac{k^2}{\alpha} c\left(\frac{x}{k}
\right)=
\left\{\begin{array}{cl}\frac{k^2}{\alpha}-\frac{k^2}{\alpha}\sqrt{1-\frac1{k^2}|x|^2} &\mbox{ if }|x|\leq k,\\
+\infty &\mbox{ if } |x|> k,\end{array}\right.
\end{equation*}
with $\alpha,k>0$ and
$$
c_{\alpha,k}^*(x)= \frac{k^2}{\alpha}\sqrt{1+\frac{\alpha^2}{k^2}|x|^2}-\frac{k^2}{\alpha}= \frac{k^2}{\alpha} c^*\left(\frac{\alpha}{k}x
\right)
$$
and $U(s)=s \log (s)$, 
leading to the flux-limited equation
\begin{equation}\label{relheat2}
\derpar{\rho}{t} =
\nabla\cdot\left(\rho\frac{\nabla\rho}{\sqrt{\frac{1}{\alpha^2}\rho^2+\frac{1}{k^2}|\nabla\rho|^2}}\right)=
\nabla\cdot\left(\rho\frac{\nabla\log\rho}{\sqrt{\frac{1}{\alpha^2}+\frac{1}{k^2}|\nabla\log\rho|^2}}\right).
\end{equation}
We remark that taking $\alpha\to \infty$ in \eqref{relheat2} leads formally to the total variation flow and taking $k\to\infty$ leads to the heat equation.

Finally, to verify that the formulation \eqref{rhonp}--\eqref{CC2} indeed 
gives us the correct discretization of the gradient flow, we rewrite it as in \cite{Brenier03} 
by enforcing the linear constraint via a Lagrange multiplier 
$\phi(x,t)$. For $\rho >0$, the problem formally becomes
$$
\min_{( {\rho}, \boldsymbol{m})} \sup_{\phi}    \int_\Omega \int_0^{\Delta t} \left(\rho\, 
c\left(\frac{\boldsymbol{m}}{\rho}\right) + \phi\left( \derpar{\rho}{t} + \nabla \cdot \boldsymbol{m}\right)\right)\rd x \rd s + 
{\cal F}(\rho(\cdot, \Delta t)).
$$
Integrating by parts, this is equivalent to
\begin{align*}
\min_{( {\rho}, \boldsymbol{m})} \sup_{\phi}   
\int_{\Omega} \int_0^{\Delta t}  \left( c\left(\frac{\boldsymbol{m}}{\rho}\right)  
\right.&\left.- \,\derpar{\phi }{t} - \nabla \phi \cdot \frac{\boldsymbol{m}}{\rho} \right)\rho\,\rd x \rd t   +  
\int_{\Omega} \phi(\cdot,\Delta t) \rho(\cdot,\Delta t)   \\ &- 
\int_{\Omega} \phi(\cdot, 0) \rho(\cdot, 0)   + {\cal F} (\rho(\cdot, 
\Delta t) ) \,.
\end{align*}
Therefore, the optimality conditions are
\begin{equation*} 
\derpar{\phi }{t} = - c^*(\nabla \phi), \quad \nabla \phi = 
\nabla c\left(\frac{\boldsymbol{m}}{\rho}\right), \quad \phi(\cdot, \Delta t) = - 
\frac{\delta {\cal F}}{ \delta \rho}(\rho(\cdot, \Delta t)).
\end{equation*}
If we substitute the optimality conditions in the continuity equation 
(\ref{CC1})-(\ref{CC2}), noting that $\nabla c^*\circ\nabla 
c$ is the identity function in $\Omega$,  we finally obtain
\begin{equation}
\label{eq:eq_opt_cond}
\derpar{\rho }{t}= -\nabla\cdot \left\{
\rho\nabla c^*\left(\nabla \phi \right)\right\}, \quad \nabla \phi(\cdot, \Delta t) = - 
\nabla U'(\rho(\cdot, \Delta t))
\end{equation}
in the interval $[0,\Delta t]$. Recalling that $c=\phi\circ |\cdot|$, the assumptions on $\phi$ imply that $\nabla c^*\left(- 
\nabla U'(\rho(\cdot, \Delta t))\right)=-\nabla c^*\left(
\nabla U'(\rho(\cdot, \Delta t))\right)$, and hence~\eqref{eq:eq_opt_cond} yields a time discrete approximation of \eqref{contgen}.

\section{Discretization of Problem}
In this section we discretize the JKO step \eqref{rhonp}--\eqref{CC2}. The problem 
reduces to: given $n\in\NN$ and an approximation $\rho^n_{\Delta t}$ of the density $\rho(\cdot,n\Delta t)$, an approximation $\rho_{\Delta t}^{n+1}$ of the density  $\rho(\cdot,(n+1)\Delta t)$  is given by $\rho^{*}$ where $(\rho^{*},\boldsymbol{m}^{*})$ minimizes
\begin{eqnarray} 
\label{e:ofdiscrete}
(\rho, \boldsymbol{m})\mapsto \Delta t \int_\Omega \Phi_c(\rho(x), \boldsymbol{m}(x))\rd x  + {\cal F}(\rho),
\end{eqnarray}
over the set of $(\rho,\boldsymbol{m})$ satisfying 
\begin{equation}
\label{CCD}
\begin{aligned}
\rho+ \Delta t\nabla \cdot \boldsymbol{m} &=  \rho^n_{\Delta t}\quad\text{ on } 
\Omega  \,,  
\\ \boldsymbol{m} \cdot \nu &= 0\quad\text{ on } \partial \Omega.
\end{aligned}
\end{equation}

In order to discretize in the space variable, for simplicity we assume in what follows that $\Omega$ is a $d$-dimensional hyperrectangle, i.e. $\Omega = \prod_{l=1}^{d} ]a_l,b_l[$. For every $l\in\{1,\hdots,d\}$, let $N_l\in\NN$, let $\Delta x_l=(b_l-a_l)/N_l$, and set $\deltaxbf =(\Delta x_1,\hdots,\Delta x_{d})$. Let us define $\I=\{i\in\NN^{d}\,:\,i_{l}\in\{0,\hdots,N_l\}\;\forall l\in\{1,\hdots,d\}\}$, $\I^{\circ}=\{i\in\I\,:\,i_{l}\neq\{0,N_l\}\;\forall\,l\in\{1,\hdots,d\}\}$, $\partial\I=\I\setminus\I^{\circ}$, and denote by $\RR^{\I}$ the set of real-valued functions defined on $\I$. We associate with $\I$ the grid points $\boldsymbol{x}_{i}\in\overline{\Omega}$ defined by $(\boldsymbol{x}_i)_{l}=a_{l}+i_{l}\Delta x_l$ for every $i\in\I$ and $l\in\{1,\hdots,d\}$. In our approximation of the optimization problem defined by the cost~\eqref{e:ofdiscrete} and the constraint \eqref{CCD}, instead of functions $(\rho,\boldsymbol{m})$ defined over $\Omega$, we consider grid functions $(\widehat{\rho},\widehat{\boldsymbol{m}})\in \RR^{\I}\times(\RR^{\I})^{d}$ and we assume that an approximation $\widehat{\rho}_{n}(\boldsymbol{x}_i)$ of $\rho^n_{\Delta t}(\boldsymbol{x}_i)$ is given for every $i\in\I$.

For notational simplicity, given $f\in\RR^{\I}$ and $i\in\I$, we will write $f_{i}$ for $f(i)$. For every $l\in\{1,\hdots,d\}$, $D_{l}^{\deltaxbf}f_{i}$ denotes the centered difference 
$$
D^{\deltaxbf}_l f_i = \frac{f_{i+{\bf e}_{l}} -f_{i-{\bf e}_l}}{2\Delta x_l}, 
$$
where ${\bf e}_{l}$ denotes the $l$th canonical vector in $\R^d$. This discrete derivative allows to approximate the divergence term in~\eqref{CCD} when applied to $\widehat{\boldsymbol{m}}=(\widehat{\boldsymbol{m}}_1,\hdots,\widehat{\boldsymbol{m}}_{d})\in (\RR^{\I})^{d}$ at $i\in\I$ by $\sum_{l=1}^{d}D^{\deltaxbf}_{l}(\widehat{\boldsymbol{m}}_{l})_{i}$. Summing up, we obtain the following discretization of~\eqref{CCD}:
\begin{equation}
\label{CCDh}
\begin{aligned} 
\widehat{\rho}_{i} + \Delta t\sum_{l=1}^{d}D^{\deltaxbf}_{l}(\widehat{\boldsymbol{m}}_{l})_{i}=\widehat{\rho}_{n}(\boldsymbol{x}_i)&\quad\text{for }i\in 
\I^{\circ}, 
 \\ 
 \sum_{l=1}^{d}(\widehat{\boldsymbol{m}}_{l})_{i}\nu_{l}(\boldsymbol{x}_i) = 0   &\quad\text{for }i\in\partial\I. 
\end{aligned} 
\end{equation}

In order to discretize the cost functional~\eqref{e:ofdiscrete}, it will be useful to extend functions defined on the grid to piecewise constant functions defined on $\overline{\Omega}$. For this purpose, let $\widehat{\I}=\{i\in\NN^{d}\,:\,i_{l}\in\{0,\hdots,N_l-1\}\;\forall l\in\{1,\hdots,d\}\}$ and notice that $\overline{\Omega}=\bigcup_{i\in\widehat{\I}}Q_{i}$, where, for every $i \in\widehat{\I}$, $Q_i=\prod_{l=1}^{d}I_{i,l}$ with
$$
I_{i,l}=\begin{cases}
[a_l+i_l(\Delta x)_l, a_l+ (i_l+1)( \Delta x)_l[&\text{if } i_l\in\{0,\hdots,N_l-2\},\\
[b_l-(\Delta x)_l, b_l]&\text{if } i_l= N_{l}-1.
\end{cases}
$$
A function $f\in\RR^{\I}$ on the grid can then be extended to a piecewise constant function on $\overline{\Omega}$ by setting 
\begin{align*} 
 f^{\deltaxbf} := \sum_{i\in\widehat{\I}}   f_{i} 1_{Q_{i}}\quad\text{where, for every $z\in\overline{\Omega}$,}\quad1_{Q_{i}}(z) = \begin{cases} 1\quad& \text{if }z \in Q_{i}, \\ 0\quad& \text{ otherwise.} \end{cases}
\end{align*}
Using this extension, we approximate the energy ${\cal F}(\rho) = \int_{\Omega} U(\rho)\dd x$ by a discrete analogous energy $\mathcal{F}^{\deltaxbf}$, which, for every $\widehat{\rho}\in\RR^{\I}$, is defined by
\begin{align*} 
 \mathcal{F}^{\deltaxbf}(\widehat{\rho}) &:=\int_{\Omega}U(\widehat{\rho}^{\deltaxbf}(x))\dd x=h\sum_{i\in\widehat{\I}}  U(\widehat{\rho}_i),
\end{align*}
where $h:=\Delta x_1\cdots\Delta x_d$.

Finally, regarding the discretization of the integral term in~\eqref{e:ofdiscrete} involving the perspective function $\Phi_{c}$, given $(\widehat{\rho},\widehat{\boldsymbol{m}})\in \RR^{\I}\times(\RR^{\I})^{d}$ we can argue as above to deduce the following first-order in time approximation 
\begin{align}
\label{def:Phi_h_c}
I_c^{\deltaxbf}(\widehat{\rho},\widehat{\boldsymbol{m}}):=\Delta t h\sum_{i\in\widehat{\I}}\Phi_c(\widehat{\rho}_{i},\widehat{\boldsymbol{m}}_{i}),
\end{align}
see \cite{CCWW,LLW} . Altogether, we are led to consider the problem of minimizing $I_c^{\deltaxbf}(\widehat{\rho},\widehat{\boldsymbol{m}})+\mathcal{F}^{\deltaxbf}(\widehat{\rho})$ over $(\widehat{\rho},\widehat{\boldsymbol{m}})\in \RR^{\I}\times(\RR^{\I})^{d}$ satisfying~\eqref{CCDh}. Notice that identifying $(\widehat{\rho},\widehat{\boldsymbol{m}})$ with a vector in $\RR^{|\I|}\times\RR^{d|\I|}$, the constraint~\eqref{CCDh} takes the form 
\begin{align}
\label{eq:linear_constraints}
\Amat (\widehat{\rho},\widehat{\boldsymbol{m}}) = b_n\,
\end{align}
for a suitable matrix $\Amat\in \RR^{|\I|\times (d+1)|\I|}$ and a vector $b_n\in\RR^{|\I|}$. Since we can only expect that the approximation of the spatial derivatives to hold up to an error term,  as in~\cite{CCWW} we relax the equality constraint~\eqref{eq:linear_constraints} in the form $|\Amat  (\widehat{\rho},\widehat{\boldsymbol{m}}) -b_n|\leq \delta$, for some $\delta\geq 0$. This consideration leads to the following discrete approximation of the JKO step:
\begin{align*}
\min\;\left\{I_c^{\deltaxbf}(\widehat{\rho},\widehat{\boldsymbol{m}})+\mathcal{F}^{\deltaxbf}(\widehat{\rho})\;:\; |\Amat (\widehat{\rho},\widehat{\boldsymbol{m}})-b_n|\leq \delta\right\},
\end{align*}
or, equivalently,
\begin{align}
\label{eq:fully_discrete_optmization_problem}
\min_{(\widehat{\rho},\widehat{\boldsymbol{m}})}I_c^{\deltaxbf}(\widehat{\rho},\widehat{\boldsymbol{m}})+\mathcal{F}^{\deltaxbf}(\widehat{\rho})+\iota_{ B_{\delta}(b_n)}(\Amat (\widehat{\rho},\widehat{\boldsymbol{m}})),
\end{align}
where $B_{\delta}(b_n)=\{z\in\RR^{|\I|}\,:\,|z-b_n|\leq\delta\}$ and, for every $C\subseteq\RR^{|\I|}$,
$$
\iota_{C}(b')=\begin{cases} 0, \quad&\text{if }b'\in C,\\ 
+\infty, \quad& \text{otherwise}. 
\end{cases}
$$
Note that, when $C$ is closed and convex, $\iota_{C}$ is lower semicontinuous and convex.

In order to provide the algorithms 
to solve \eqref{eq:fully_discrete_optmization_problem}, we need to introduce
the proximity operator of a proper convex lower semicontinuous function 
$f\colon\RR^M\to\RX$, which is defined by 
\begin{equation*}
(\forall x\in\RR^M)\quad \prox_{f}x=
\arg\min_{y\in\RR^M}\Big(f(y)+\frac{1}{2}|y-x|^2\Big).
\end{equation*}
Since, for every $x\in\RR^M$, the function 
$f+|\cdot-x|^2/2$ is strongly convex and lower semicontinuous, $\prox_f$ is well defined. 
Then, under the assumption that $\mathcal{F}^{\deltaxbf}$ is differentiable with Lipschitz gradient, by defining $\Psi\colon(\widehat{\rho},\widehat{\boldsymbol{m}})\mapsto \mathcal{F}^{\deltaxbf}(\widehat{\rho})$, the algorithm developed in \cite{Yan} applied to~\eqref{eq:fully_discrete_optmization_problem} becomes 
\begin{align}
   & \phi^{(\ell+1)} = \prox_{\sigma i_{B_\delta(b_n)}^*} (\phi^{(\ell)} + \sigma \Amat 
   \bar{{u}}^{(\ell)}),
   \label{eq:Yan_step_1}
    \\ & {u}^{(\ell+1)} = \prox_{\lambda I_c^{\deltaxbf}} 
    ({u}^{(\ell)} - \lambda\nabla\Psi(u^{(\ell)})- \lambda  \Amat^\top 
    \phi^{(\ell+1)}),\label{eq:Yan_step_2}
    \\ & \bar{{u}}^{(\ell+1)} = 2{u}^{(\ell+1)} - {u}^{(\ell)} + \lambda  \nabla\Psi(u^{(\ell)})- \lambda \nabla\Psi(u^{(\ell+1)}),
    \label{eq:Yan_step_3}
\end{align}
where $\sigma,\,\lambda\in\RPP$ and $u^{(\ell)}=(\widehat{\rho}^{(\ell)},\widehat{\boldsymbol{m}}^{(\ell)})$ and $\phi^{(\ell)}$ denote the primal and dual iterates, respectively. The convergence of $(\widehat{\rho}^{(\ell)},\widehat{\boldsymbol{m}}^{(\ell)})$, as $\ell\to\infty$, towards a solution $(\widehat{\rho}_{n+1},\widehat{\boldsymbol{m}}_{n+1})$ of~\eqref{eq:fully_discrete_optmization_problem} is ensured if $\sigma\lambda\|\Amat\|^2<1$ and $\lambda<2/L$,
where $L>0$ is the Lipschitz constant of $\nabla {\mathcal F^{\deltaxbf}}$. In addition, the algorithm proposed in \cite{condat,vu} reads
\begin{align}
   & \phi^{(\ell+1)} = \prox_{\sigma i_{B_\delta(b_n)}^*} (\phi^{(\ell)} + \sigma \Amat 
   \bar{{u}}^{(\ell)}),
   \label{eq:Condat_Vu_step_1}
    \\ & {u}^{(\ell+1)} = \prox_{\lambda I_c^{\deltaxbf}} ({u}^{(\ell)} - 
    \lambda \nabla \Psi(u^{(\ell)}) - \lambda  \Amat^\top     \phi^{(\ell+1)}),
    \label{eq:Condat_Vu_step_2}
    \\ & \bar{{u}}^{(\ell+1)} = 2{u}^{(\ell+1)} - {u}^{(\ell)},
    \label{eq:Condat_Vu_step_3}
\end{align}
which converges under the condition $\sigma\|\Amat\|^2\le\frac{1}{\lambda}-\frac{L}{2}$.

Both procedures \eqref{eq:Yan_step_1}-\eqref{eq:Yan_step_3} and \eqref{eq:Condat_Vu_step_1}-\eqref{eq:Condat_Vu_step_3} produce an approximation $\widehat{\rho}_{n+1}$ of $\rho^{n+1}_{\Delta t}$ which is then used to solve the next JKO step. In both algorithms, it is necessary to compute 
$\prox_{\lambda I_c^{\deltaxbf}}$, where $\lambda>0$. It follows from~\eqref{def:Phi_h_c} that
$$
(\forall\,i\in\I)\quad (\prox_{\lambda I_c^{\deltaxbf}}(u))_{i}=\prox_{\gamma \Phi_{c}}(u_{i}),
$$
where $\gamma=\lambda \Delta t h$. In particular, steps~\eqref{eq:Yan_step_2} and~\eqref{eq:Condat_Vu_step_2} can be parallelized and they only involve the computation of $\prox_{\gamma\Phi_{c}}$, which is addressed in the following section. We also study a more compact version of these algorithms, where the term $\nabla \Psi$ is incorporated directly in a joint proximal operator $\prox_{\gamma(\Phi_c+\Psi)}$. This formula is useful when $\nabla\Psi$ is not Lipschitz or its constant is not easy to compute.

\section{Algorithms and proximity operator calculus}
Before the computation of the proximity operator of $\Phi$, let us first recall some of its properties. Given a proper convex lower semicontinuous function 
$f\colon\RR^M\to\RX$
it follows from~\cite[Proposition~16.44]{BausComb19}, that
\begin{equation}
\label{e:caracprox}
(\forall (x,p)\in\RR^M\times\RR^M)\quad
p=\prox_{f}x\:\:\Leftrightarrow\:\:
x-p\in \partial f(p),
\end{equation}
where
\begin{equation*}
\label{e:defsubdif}
\partial f\colon\RR^M\to 2^{\RR^M}\colon x\mapsto 
\{u\in\RR^M\,|\,(\forall y\in\RR^M)\quad f(x)+u^{\top}(y-x)\leq f(y)\}
\end{equation*}
is the subdifferential of $f$.
In addition, from~\cite[Proposition~16.10]{BausComb19}, we have
\begin{equation}
\label{e:FY}
(\forall (x,u)\in\RR^M\times\RR^M)\quad
u\in\partial f(x)\quad\Leftrightarrow\quad 
f(x)+f^*(u)=u^{\top}x,
\end{equation}
where we recall that $f^*$ denotes the Fenchel conjugate of $f$. 
Moreover, by~\cite[Proposition~14.3{\rm(ii)}]{BausComb19}, for 
every $\gamma>0$, we have
\begin{equation}
\label{e:Moreau}
\prox_{\gamma f^*}=\id-\gamma \prox_{f/\gamma}\circ(\id/\gamma).
\end{equation}
Recall that the perspective function $\Phi_{c}$ is defined in~\eqref{Jdefc} and we assume that $c$ has the form $c=\phi\circ|\cdot|$. Now we provide the proximity operator of $\Phi_c$ in a general context
which includes the relativistic heat and $p$-Laplacian cases. 

\begin{prop}
\label{p:proxgen}
Let $F\colon\RR\to\left]-\infty,+\infty\right]$ be a lower semicontinuous convex function which is differentiable in $\inte\dom F=\RPP$ and such that $L_0:=\lim_{\xi\downarrow 0}F'(\xi)\in\RR\cup\{-\infty\}$.  Moreover, let 
$\Psi\colon\RR\times\RR^d\to\left]-\infty,+\infty\right]\colon (\rho,\boldsymbol{m})\mapsto 
F(\rho)$ and let
$\phi\colon\RR\to\left]-\infty,+\infty\right]$ be even, convex, lower 
semicontinuous,
differentiable in $\inte\dom\phi$, such that $\phi(\xi)=0$ if and only if $\xi=0$, and one of 
the following holds: 
\begin{enumerate}
\item \label{p:proxgeni}
$\dom\phi=\RR$,
$\lim_{\xi\to-\infty}\phi'(\xi)=-\infty$, and $\lim_{\xi\to+\infty}\phi'(\xi)=+\infty$ or
\item \label{p:proxgenii}
$\dom\phi=[-a,a]$, $\lim_{\xi\to -a}\phi'(\xi)=-\infty$, and $\lim_{\xi\to a}\phi'(\xi)=+\infty$, for some 
$a\in\RPP$.
\end{enumerate}
Set 
\begin{equation}
\label{e:auxgen}
\Phi\colon\RR\times\RR^d\to\RX\colon
(\rho,\boldsymbol{m})\mapsto
\begin{cases}
\rho\,\phi\left(\dfrac{|\boldsymbol{m}|}{\rho}\right),\quad&\text{if}\:\:\rho>0;\\
0,&\text{if}\;\;(\rho,\boldsymbol{m})=(0,0);\\
+\infty,&\text{otherwise}.
\end{cases}
\end{equation}
Then, for every $\gamma>0$, we have
\begin{equation*}
\prox_{\gamma(\Phi+\Psi)}\colon 
(\rho,\boldsymbol{m})\mapsto
\begin{cases}
(0,0),&\text{if}\:\:\rho+\gamma\phi^*\left(\dfrac{|\boldsymbol{m}|}{\gamma}\right)\le \gamma L_0;\\[3mm]
(\theta,\boldsymbol{v}),&\text{if}\:\:\rho+\gamma\phi^*\left(\dfrac{|\boldsymbol{m}|}{\gamma}\right)>\gamma L_0,
\end{cases}
\end{equation*}
where $\theta\ge\prox_{\gamma F}\rho\ge 0$ is the 
unique strictly positive solution to $R^{\phi}_{\rho,\boldsymbol{m}}(t)=0$ 
with
\begin{align}
\label{e:nonlineareqgen}
R^{\phi}_{\rho,\boldsymbol{m}}\colon t\mapsto t+\gamma 
F'(t)-\rho-\gamma\phi^*\bigg(\prox_{\frac{t}{\gamma}\phi^*}\bigg(\frac{|\boldsymbol{m}|}{\gamma}\bigg)\bigg),
\end{align} 
and 
\begin{equation}
\label{e:vhat}
{\boldsymbol{v}}=\begin{cases}
\left(1-\dfrac{\gamma}{|\boldsymbol{m}|}\prox_{\frac{\theta}{\gamma}\phi^*}
\left(\dfrac{|\boldsymbol{m}|}{\gamma}\right)\right)\boldsymbol{m}
,&\text{if}\;\;\boldsymbol{m}\neq 0;\\
0,&\text{if}\;\;\boldsymbol{m}= 0.
\end{cases}
\end{equation}
\end{prop}
\begin{proof}
First note that the hypotheses on $\phi$ imply 
$\phi\ge 0$. Then, since
$\phi(0)=0$, we have
\begin{equation}
\label{e:phistarpos}
(\forall \theta\in\RR)\quad 
\phi^*(\theta)=\sup_{\rho\in\RR}\big(\theta\rho-\phi(\rho)\big)\ge 
-\phi(0)=0=\phi^*(0).
\end{equation}
Moreover, note that, under \ref{p:proxgeni} and \ref{p:proxgenii} we have $\dom\phi^*=\RR$ and
\begin{equation*}
(\forall \xi\in\RR)\quad 
\phi^{\infty}(\xi):=\lim_{\lambda\to+\infty}\frac{\phi(\lambda\xi)}{\lambda}=\iota_{\{0\}}(\xi).
\end{equation*}
Then, since $(\phi\circ|\cdot|)^{\infty}=\phi^{\infty}\circ|\cdot|$
and $(\phi\circ|\cdot|)^{*}=\phi^{*}\circ|\cdot|$ \cite[Example~13.8]{BausComb19}, we 
deduce that $\Phi$ defined in \eqref{e:auxgen} is a proper convex 
lower semicontinuous function and that 
\begin{equation}
\label{e:conjpersp}
\Phi^*=\iota_C,\quad\text{where}\quad C=\menge{(\rho,\boldsymbol{m})\in\RR\times\RR^d}{\rho+\phi^*(|\boldsymbol{m}|)\le 0},
\end{equation}
in view of 
\cite[Theorem~3F]{Rock66} and \cite[Example~13.8]{BausComb19}. Therefore, $\Phi+\Psi$ is
proper, convex, and lower semicontinuous and
it follows from \eqref{e:caracprox} and \cite[Corollary~16.48(iv)]{BausComb19} that, given $(\theta,\boldsymbol{v})$ and 
$(\rho,\boldsymbol{m})$ in $\RR\times\RR^d$, we have
\begin{equation}
\label{e:proxmain}
(\theta,\boldsymbol{v})=\prox_{\gamma 
(\Phi+\Psi)}(\rho,\boldsymbol{m})\quad\Leftrightarrow\quad 
\frac{1}{\gamma}
\begin{pmatrix}
\rho-\theta\\
\boldsymbol{m}-\boldsymbol{v}
\end{pmatrix}\in \partial\Phi(\theta,\boldsymbol{v})+\partial F(\theta)\times \{\boldsymbol{0}\}.
\end{equation}
Since \eqref{e:caracprox} implies that 
$(\theta,\boldsymbol{v})\in\dom\partial\Phi\subset\dom\Phi$, we split the proof in two parts by supposing that 
either $\theta=0$ or $\theta>0$. Moreover, note that
\begin{equation*}
\label{e:subdif0}
\partial F(0)=
\begin{cases}
\left]-\infty,L_0\right],&\text{if}\:\: L_0\in\RR;\\
\varnothing, &\text{if}\:\: L_0=-\infty
\end{cases}
\end{equation*}
and, hence, in the case when $L_0=-\infty$, only the case $\theta>0$ remains.
\begin{itemize}
 \item[$\theta=0$]: We deduce from \eqref{e:auxgen} that $\boldsymbol{v}=0$ and, hence, since $L_0\in\partial F(0)$, it follows from \eqref{e:proxmain}, \eqref{e:conjpersp}, and \eqref{e:FY} that 
\begin{align*}
(0,\boldsymbol{0})=\prox_{\gamma 
(\Phi+\Psi)}(\rho,\boldsymbol{m})\quad &\Leftrightarrow\quad \left(\frac{\rho}{\gamma}-L_0,\frac{\boldsymbol{m}}{\gamma}\right)\in\partial\Phi(0,\boldsymbol{0})\nonumber\\
&\Leftrightarrow\quad 
\Phi(0,\boldsymbol{0})+\Phi^*\left(\frac{\rho}{\gamma}-L_0,\frac{\boldsymbol{m}}{\gamma}\right)=0\nonumber\\
&\Leftrightarrow\quad 
\left(\frac{\rho}{\gamma}-L_0,\frac{\boldsymbol{m}}{\gamma}\right)\in C\nonumber\\
&\Leftrightarrow\quad 
\rho+\gamma\phi^*\bigg(\frac{|\boldsymbol{m}|}{\gamma}\bigg)\le 
\gamma L_0.
\end{align*}

\item[$\theta>0$]: Note that, in this case $F$ is differentiable at $\theta$.
Hence, if $\boldsymbol{v}=0$ we have from \eqref{e:proxmain}, \eqref{e:FY}, \eqref{e:conjpersp}, and \eqref{e:phistarpos} that  
\begin{align*}
&(\theta,\boldsymbol{0})=\prox_{\gamma 
(\Phi+\Psi)}(\rho,\boldsymbol{m})\nonumber\\
&\hspace{2cm}
\Leftrightarrow\:\:\Phi(\theta,\boldsymbol{0})+\Phi^*\left(\frac{\rho-\theta}{\gamma}- 
F'(\theta),\frac{\boldsymbol{m}}{\gamma}\right)=\theta\left(\frac{\rho-\theta}{\gamma}- 
F'(\theta)\right)
\nonumber\\
&\hspace{2cm}\Leftrightarrow\:\: \rho=\theta+\gamma F'(\theta)\quad\text{and}\quad \frac{\rho-\theta}{\gamma}- 
F'(\theta)+\phi^*\left(\frac{|\boldsymbol{m}|}{\gamma}\right)\le 0
\nonumber\\
&\hspace{2cm}\Leftrightarrow\:\: \theta=\prox_{\gamma F}(\rho)\quad\text{and}\quad \boldsymbol{m}=0.
\end{align*}

If $\boldsymbol{v}\ne0$ we have 
that $\Phi$ is differentiable at $(\theta,\boldsymbol{v})$ and we deduce from \eqref{e:proxmain} 
and \eqref{e:FY} that
\begin{align}
\label{e:auxthetapos}
(\theta,\boldsymbol{v})=\prox_{\gamma (\Phi+\Psi)}(\rho,\boldsymbol{m})\quad
&\Leftrightarrow\quad 
\begin{cases}
\frac{1}{\gamma}(\rho-\theta-\gamma F'(\theta))=
\phi\big(\frac{|\boldsymbol{v}|}{\theta}\big)-\phi'\big(\frac{|\boldsymbol{v}|}{\theta}\big)\big(\frac{|\boldsymbol{v}|}{\theta}\big)\\
\frac{1}{\gamma}(\boldsymbol{m}-\boldsymbol{v})= \phi'\big(\frac{|\boldsymbol{v}|}{\theta}\big)\frac{\boldsymbol{v}}{|\boldsymbol{v}|}
\end{cases}\nonumber\\
&\Leftrightarrow\quad 
\begin{cases}
\rho-\theta-\gamma F'(\theta)= 
-\gamma\phi^*\big(\phi'\big(\frac{|\boldsymbol{v}|}{\theta}\big)\big)\\
\boldsymbol{m}-\boldsymbol{v}=\gamma\phi'\big(\frac{|\boldsymbol{v}|}{\theta}\big)\frac{\boldsymbol{v}}{|\boldsymbol{v}|}.
\end{cases}
\end{align}
Note that, \eqref{e:auxthetapos} and \eqref{e:phistarpos} imply
\begin{equation*}
\theta+\gamma 
F'(\theta)-\rho=\gamma\phi^*\Big(\phi'\Big(\frac{|\boldsymbol{v}|}{\theta}\Big)\Big)\ge
 0.
\end{equation*}
Hence, since $\id+F'$ is strictly increasing, so is its inverse and, since $\dom F\subset \RP$, we deduce
\begin{equation}
\label{e:lowboundthet}
\theta\ge\prox_{\gamma F}\rho\ge 0.
\end{equation}
In addition, we obtain from \eqref{e:auxthetapos} that
\begin{equation}
\label{e:vfctm}
\boldsymbol{v}=\frac{\boldsymbol{m}}{1+\gamma\phi'\big(\frac{|\boldsymbol{v}|}{\theta}\big)/|\boldsymbol{v}|}=\frac{|\boldsymbol{v}|}{|\boldsymbol{v}|+\gamma\phi'\big(\frac{|\boldsymbol{v}|}{\theta}\big)}\boldsymbol{m},
\end{equation}
which yields
\begin{equation*}
\label{e:proxphiaux}
\frac{|\boldsymbol{m}|-|\boldsymbol{v}|}{\gamma}=\phi'\Big(\frac{|\boldsymbol{v}|}{\theta}\Big)
\end{equation*}
or, equivalently,
\begin{equation}
\label{e:normv}
|\boldsymbol{v}|=\theta\,\prox_{\frac{\gamma}{\theta}\phi}\Big(\frac{|\boldsymbol{m}|}{\theta}\Big).
\end{equation}
Thus, it follows from \eqref{e:Moreau} that
\begin{equation}
\label{e:phiprime}
\phi'\Big(\frac{|\boldsymbol{v}|}
{\theta}\Big)=\frac{|\boldsymbol{m}|-|\boldsymbol{v}|}{\gamma}=\prox_{\frac{\theta}{\gamma}\phi^*}\Big(\frac{|\boldsymbol{m}|}{\gamma}\Big).
\end{equation}
Hence, we deduce from 
\eqref{e:auxthetapos} that
\begin{equation*}
R^{\phi}_{\rho,\boldsymbol{m}}(\theta)=\theta+\gamma 
F'(\theta)-\rho-\gamma\phi^*\bigg(\prox_{\frac{\theta}{\gamma}\phi^*}\Big(\frac{|\boldsymbol{m}|}{\gamma}\Big)\bigg)=0.
\end{equation*}
Note that, since $\dom\phi^*=\RR$, \cite[Lemma~3.3(iii)]{BCS2} implies that 
\begin{equation*}
\phi^*\bigg(\prox_{\frac{t}{\gamma}\phi^*}\Big(\frac{|\boldsymbol{m}|}{\gamma}\Big)\bigg)\to
 \phi^*\Big(\frac{|\boldsymbol{m}|}{\gamma}\Big)\quad\text{as}\:\: t\to0
\end{equation*}
 and, hence,
$\lim_{t\downarrow 0}R^{\phi}_{\rho,\boldsymbol{m}}(t)=\gamma L_0-\rho-\gamma \phi^*(|\boldsymbol{m}|/\gamma)<0$.
Moreover, since \cite[Proposition~12.27]{BausComb19} implies 
that $R^{\phi}_{\rho,\boldsymbol{m}}$ is strictly increasing,
$$\lim_{t\to+\infty}R^{\phi}_{\rho,\boldsymbol{m}}(t)=+\infty,$$ and \cite[Corollary~17.43]{BausComb19} implies the continuity of $R^{\phi}_{\rho,\boldsymbol{m}}$,
we verify from 
\eqref{e:lowboundthet} that there exists a unique solution 
$\theta\ge\prox_{\gamma F}\rho$ to the 
nonlinear equation $R^{\phi}_{\rho,\boldsymbol{m}}(t)=0$.
Finally, note that \eqref{e:normv}, \eqref{e:phiprime}, and 
\eqref{e:Moreau} yield
\begin{align*}
\label{e:constwm}
\frac{|\boldsymbol{v}|}{|\boldsymbol{v}|+\gamma\phi'(\frac{|\boldsymbol{v}|}{\theta})}&=\frac{\theta\prox_{\frac{\gamma}{\theta}\phi}(|\boldsymbol{m}|/\theta)}{\theta\prox_{\frac{\gamma}{\theta}\phi}(|\boldsymbol{m}|/\theta)+\gamma\prox_{\frac{\theta}{\gamma}\phi^*}(|\boldsymbol{m}|/\gamma)}\nonumber\\
&= \frac{\theta\prox_{\frac{\gamma}{\theta}\phi}(|\boldsymbol{m}|/\theta)}{|\boldsymbol{m}|}\nonumber\\
&=1-\frac{\gamma}{|\boldsymbol{m}|}\prox_{\frac{\theta}{\gamma}\phi^*}\Big(\frac{|\boldsymbol{m}|}{\gamma}\Big),
\end{align*}
and the result follows from \eqref{e:vfctm}.
\end{itemize}
The proof is complete.
\end{proof}
\begin{remark}\
\label{r:1}
\begin{enumerate}
\item Note that, under the hypotheses on $\phi$ in 
Proposition~\ref{p:proxgen}, \eqref{e:phistarpos} and Fermat's 
theorem imply that $0\in\partial\phi^*(0)$ and hence, for every $\lambda>0$, we have $0=\prox_{\lambda\phi^*}0$.
Therefore, since, for every $\lambda>0$, $\prox_{\lambda \phi^*}$ 
is increasing and nonexpansive 
\cite[Proposition~24.31]{BausComb19}, we have, for every 
$\theta>0$
and $\gamma>0$,
\begin{equation}
\label{e:proxpos}
0=\prox_{\frac{\theta}{\gamma}\phi^*}(0)\le 
\prox_{\frac{\theta}{\gamma}\phi^*}\Big(\frac{|\boldsymbol{m}|}{\gamma}\Big)
\le 
\frac{|\boldsymbol{m}|}{\gamma}.
\end{equation}
Hence, the constant multiplying $\boldsymbol{m}$ in \eqref{e:vhat} is nonnegative.

\item In next section, the following nontrivial examples for $F$ satisfying the hypotheses of Proposition~\ref{p:proxgen} are used:
\begin{enumerate}
\item Set 
\begin{equation*}
F\colon \xi\mapsto
\begin{cases}
\xi\ln \xi,&\text{if}\:\: \xi>0;\\
0,&\text{if}\:\: \xi=0;\\
+\infty,&\text{if}\:\: \xi<0.
\end{cases}
\end{equation*}
Then $F$ is convex differentiable in $\RPP=\inte\dom F$,
$F'\colon ]0,+\infty[\to \RR$ is given by $F'\colon \xi\mapsto \ln(\xi)+1$, and $L_0=-\infty$.

\item Let $\eta>0$ with $\eta\neq 1$ and set 
\begin{equation*}
F\colon \xi\mapsto
\begin{cases}
\dfrac{\xi^{\eta}}{\eta(\eta-1)},&\text{if}\:\: \xi\ge0;\\
+\infty,&\text{if}\:\: \xi<0.
\end{cases}
\end{equation*}
$F$ is convex differentiable in $\RPP=\inte\dom F$, $F'\colon ]0,+\infty[\to \RR$ is given by $F'\colon \xi\mapsto \xi^{\eta-1}/(\eta-1)$, 
and
\begin{equation*}
L_0=
\begin{cases}
-\infty,&\text{if}\:\:\eta<1;\\
0,&\text{if}\:\:\eta>1.
\end{cases}
\end{equation*}
\end{enumerate}

\item Moreover, since, for every $\xi\in\RR$ and $\gamma>0$, 
$\phi^*(\prox_{\gamma\phi^*}(\xi))\le \phi^*(\xi)$ 
\cite[Proposition~12.27]{BausComb19}, it follows from 
\eqref{e:nonlineareqgen} that, for every 
$\theta\ge\prox_{\gamma F}\rho$,
\begin{equation}
\label{e:upperb}
R^{\phi}_{\rho,\boldsymbol{m}}(\theta)\ge \theta+\gamma 
F'(\theta)-\rho-\gamma\phi^*\bigg(\frac{|\boldsymbol{m}|}{\gamma}\bigg).
\end{equation}
Hence, since the right hand side of \eqref{e:upperb}
is zero for $$\overline{\theta}=\prox_{\gamma 
F}\big(\rho+\gamma\phi^*(|\boldsymbol{m}|/\gamma)\big),$$ we have
that the unique solution $\theta>0$ to the scalar equation $R^{\phi}_{\rho,\boldsymbol{m}}(t)=0$ satisfies $\theta\in[\prox_{\gamma 
F}\rho,\overline{\theta}]$. This scalar equation can be solved by standard root-finding methods \cite{Pres07}.

\end{enumerate}
\end{remark}

\begin{prop}
\label{p:proxPhic}
Let $F\colon\RR\to\left]-\infty,+\infty\right]$ be a lower semicontinuous convex function which is differentiable in $\inte\dom F=\RPP$ and such that $L_0:=\lim_{\xi\downarrow 0}F'(\xi)\in\RR\cup\{-\infty\}$.  Moreover, let 
$\Psi\colon\RR\times\RR^d\to\left]-\infty,+\infty\right]\colon (\rho,\boldsymbol{m})\mapsto 
F(\rho)$, set $\alpha>0$, $k>0$, and
\begin{equation*}
\Phi^{\alpha,k}_c\colon (\rho,\boldsymbol{m})\mapsto \begin{cases}
\frac{k^2}{\alpha}(\rho-\sqrt{\rho^2-|\frac{\boldsymbol{m}}{k}|^2}),\quad&\text{if}\:\:|\frac{\boldsymbol{m}}{k}|\le\rho;\\
+\infty,&\text{otherwise}.
\end{cases}
\end{equation*}
Then, for every $\gamma>0$, we have
\begin{equation}
\label{e:calcproxPhicr}
\prox_{\gamma(\Phi^{\alpha,k}_c+\Psi)}\colon 
(\rho,\boldsymbol{m})\mapsto
\begin{cases}
(0,0),&\text{if}\:\:\rho+\sqrt{\big(\frac{\gamma k^2}{\alpha}\big)^2+k^2|\boldsymbol{m}|^{2}}-\frac{\gamma k^2}{\alpha}
\le\gamma L_0;\\[2mm]
(\theta,\boldsymbol{v}),&\text{if}\:\:
\rho+\sqrt{\big(\frac{\gamma k^2}{\alpha}\big)^2+k^2|\boldsymbol{m}|^{2}}-\frac{\gamma k^2}{\alpha}
>\gamma L_0,
\end{cases}
\end{equation}
where $\theta\ge\prox_{\gamma F}\rho\ge 0$ is the 
unique solution to $R^c_{\rho,\boldsymbol{m}}(t)=0$ with
\begin{equation*}
\label{e:nonlineareq1Fc}
R^c_{\rho,\boldsymbol{m}}\colon
\!t\mapsto\!\!\left(\!\dfrac{(1+k^2)t+\!\gamma 
F'(t)\!-\!\rho+\!\frac{\gamma k^2}{\alpha}}{t+\gamma 
F'(t)-\rho+\frac{\gamma k^2}{\alpha}}\!\right)
\!\!\sqrt{\Big(t+\gamma F'(t)\!-\!\rho+\frac{\gamma k^2}{\alpha}\Big)
^{\!2}\!-\!\Big(\frac{\gamma k^2}{\alpha}\Big)^{\!2}}-k|\boldsymbol{m}|
\end{equation*} 
and 
\begin{equation}
\label{e:vhatc}
\boldsymbol{v}=
\left(\frac{\theta k^2}{(1+k^2)\theta+\gamma 
F'(\theta)-\rho+\frac{\gamma k^2}{\alpha}}\right)\boldsymbol{m}.
\end{equation}
\end{prop}
\begin{proof}
Set
\begin{equation}
\label{e:defphiak}
\phi_{\alpha,k}\colon \xi\mapsto 
\begin{cases}
\frac{k^2}{\alpha}\Big(1-\sqrt{1-\big(\frac{\xi}{k}\big)^2}\Big),\quad&\text{ if }\:\:|\xi|\le k;\\
+\infty,&\text{ otherwise}.
\end{cases}
\end{equation}
Then, $\dom\phi_{\alpha,k}=[-k,k]$, $\phi_{\alpha,k}$ is proper, even, convex, lower 
semicontinuous, differentiable in 
$]-k,k[$, $\phi_{\alpha,k}(0)=0$, and, for every $\xi\in\left]-k,k\right[$, 
$\phi_{\alpha,k}'(\xi)=\xi/(\alpha\sqrt{1-(\xi/k)^2})$, leading to 
$\lim_{\xi\to\pm k}\phi_{\alpha,k}'(\xi)=\pm\infty$.
Moreover, by noting that 
\begin{equation}
\label{e:aux44}
\Phi_c\colon(\rho,\boldsymbol{m})\mapsto
\begin{cases}
\rho\,\phi_{\alpha,k}\left(\dfrac{|\boldsymbol{m}|}{\rho}\right),&\text{if}\:\:\rho>0;\\
0,&\text{if}\:\:\rho=0\:\:\text{and}\:\: \boldsymbol{m}=0;\\
+\infty,&\text{otherwise},
\end{cases}
\end{equation}
Proposition~\ref{p:proxgen} and
\begin{equation}
\label{e:defphiaks}
\phi_{\alpha,k}^*\colon s\mapsto \sqrt{\Big(\frac{k^2}{\alpha}\Big)^2+(k s)^2}-\frac{k^2}{\alpha}
\end{equation}
yield \eqref{e:calcproxPhicr},
where $\theta\ge\prox_{\gamma F}\rho$ is the 
unique solution to $R^{\phi_{\alpha,k}}_{\rho,\boldsymbol{m}}(t)=0$, with $R^{\phi_{\alpha,k}}_{\rho,\boldsymbol{m}}$ being defined in \eqref{e:nonlineareqgen}. Hence, we have
\begin{equation}
\label{e:sqrteqr}
\theta+\gamma 
F'(\theta)-\rho+\frac{\gamma k^2}{\alpha}=\sqrt{\Big(\frac{\gamma k^2}{\alpha}\Big)^2+(k\chi)^2},
\end{equation}
where
\begin{align}
\label{e:nu2r}
\chi=\gamma\prox_{\frac{{\theta}}{\gamma}\phi_{\alpha,k}^*}\bigg(\frac{|\boldsymbol{m}|}{\gamma}\bigg)
\quad &\Leftrightarrow\quad
\frac{|\boldsymbol{m}|-\chi}{{\theta}}=(\phi_{\alpha,k}^*)'(\chi/\gamma)
=\frac{\chi k^2}{\sqrt{\big(\frac{\gamma k^2}{\alpha}\big)^2+(k\chi)^2}}\nonumber\\
&\Leftrightarrow\quad
\chi=\frac{\sqrt{\big(\frac{\gamma k^2}{\alpha}\big)^2+(k\chi)^2}}{\theta k^2+\sqrt{\big(\frac{\gamma k^2}{\alpha}\big)^2+(k\chi)^2}
}|\boldsymbol{m}|.
\end{align}
Observe that, in view of \eqref{e:proxpos} in Remark~\ref{r:1}, 
$\chi\ge 0$, which allows us to write \eqref{e:sqrteqr} equivalently as
\begin{equation}
\label{e:nu3r}
k\chi=\sqrt{\Big(\theta+\gamma 
F'({\theta})-\rho+\frac{\gamma k^2}{\alpha}\Big)^2-\Big(\frac{\gamma k^2}{\alpha}\Big)^2}.
\end{equation}
Therefore, combining 
\eqref{e:nu3r}, \eqref{e:nu2r}, and \eqref{e:sqrteqr}, we obtain 
$R^{\phi_{\alpha,k}}_{\rho,\boldsymbol{m}}({\theta})=R^c_{\rho,\boldsymbol{m}}({\theta})=0$. 
Finally, combining \eqref{e:nu2r}, \eqref{e:sqrteqr}, and 
\eqref{e:vhat} we obtain
\eqref{e:vhatc}.
\end{proof}

In view of Remark~\ref{r:1}, an upper bound for ${\theta}$ 
in the previous case is
$$\overline{\theta}=\prox_{\gamma 
F}\left(\rho+\sqrt{\Big(\frac{\gamma k^2}{\alpha}\Big)^2+k^2|\boldsymbol{m}|^2}-\frac{\gamma k^2}{\alpha}\right).$$

\begin{remark}

\

\begin{enumerate}
\item 
\label{r:i}
Suppose that $F=\iota_{\RP}$. Then, $F$ is differentiable in $\RPP$, for every $\xi>0$, $F'(\xi)=0$, $\lim_{\xi\downarrow 0}F'(\xi)=0$,
$\prox_{\gamma F}\colon \rho\mapsto\max\{0,\rho\}$, 
$\Phi^{\alpha,k}_c+F=\Phi^{\alpha,k}_c$, and  we deduce 
\begin{equation}
\label{e:proxF0}
\prox_{\gamma\Phi^{\alpha,k}_c}\colon 
(\rho,\boldsymbol{m})\mapsto
\begin{cases}
(0,0),&\text{if}\:\:\rho+\sqrt{\big(\frac{\gamma k^2}{\alpha}\big)^2+k^2|\boldsymbol{m}|^{2}}-\frac{\gamma k^2}{\alpha}\le 0;\\
({\theta},\boldsymbol{v}),&\text{if}\:\:\rho+\sqrt{\big(\frac{\gamma k^2}{\alpha}\big)^2+k^2|\boldsymbol{m}|^{2}}-\frac{\gamma k^2}{\alpha}>0,
\end{cases}
\end{equation}
where ${\theta}\ge\max\{0,\rho\}$ is the unique solution to 
$R^c_{\rho,\boldsymbol{m}}(t)=0$ 
with
\begin{align}
\label{e:R0}
R^c_{\rho,\boldsymbol{m}}\colon
t&\mapsto 
\left(\dfrac{(1+k^2)t-\rho+\frac{\gamma k^2}{\alpha}}{t-\rho+\frac{\gamma k^2}{\alpha}}\right)
\sqrt{\Big(t-\rho+\frac{\gamma k^2}{\alpha}\Big)^2-\Big(\frac{\gamma k^2}{\alpha}\Big)^2}-k|\boldsymbol{m}|
\end{align} 
and 
\begin{equation}
\boldsymbol{v}=
\left(\frac{\theta k^2}{(1+k^2)\theta-\rho+\frac{\gamma k^2}{\alpha}}\right)\boldsymbol{m}.
\end{equation}

This particular instance can also be derived from \cite[Theorem~3.1]{CoJMAA} and \cite[Proposition~2.3]{CoEJS}.

\item Set $F=\iota_{\RP}$, consider the function $\phi_{\alpha,k}$ defined in \eqref{e:defphiak}, and its conjugate $\phi_{\alpha,k}^*$ in \eqref{e:defphiaks}.
Then, we have
\begin{equation}
\label{e:limitalpha}
\widetilde{\phi}_k:=\lim_{\alpha\to+\infty}\phi_{\alpha,k}=
\iota_{[-k,k]}\quad\text{and}\quad \widetilde{\phi}^*_k=\lim_{\alpha\to+\infty}\phi_{\alpha,k}^*=k|\cdot|,
\end{equation}
pointwise and, from \eqref{e:aux44}, we obtain
$$\widetilde{\Phi}^k_{c}(\rho,\boldsymbol{m}):=\lim_{\alpha\to+\infty}\Phi^{\alpha,k}_c(\rho,\boldsymbol{m})
=\begin{cases}
0,&\text{if}\:\:|\boldsymbol{m}|\le k\rho;\\
+\infty,&\text{otherwise}.
\end{cases}$$
Note that, \eqref{e:nu2r} is no longer valid in the limit case when $\alpha\to+\infty$. Indeed, 
it follows from \eqref{e:limitalpha} that $\widetilde{\phi}_{k}^*$ is no longer differentiable and, as a consequence, \eqref{e:R0} is not valid in the limit case. By recalculating from 
\eqref{e:nu2r} using the subdifferential of 
$\widetilde{\phi}_{k}^*$, we obtain 
\begin{equation*}
\label{e:proxF0kinf2}
\prox_{\gamma\widetilde{\Phi}^k_c}\colon 
(\rho,\boldsymbol{m})\mapsto
\begin{cases}
(0,0),&\text{if}\:\:\rho+k|\boldsymbol{m}|\le 0;\\
(\rho,\boldsymbol{m}),&\text{if}\:\:\rho+k|\boldsymbol{m}|>0\:\:
\text{and}\:\:|\boldsymbol{m}|\le k\rho;\\
\big(\frac{\rho+k|\boldsymbol{m}|}{1+k^2}\big)\Big(1,\frac{k\boldsymbol{m}}{|\boldsymbol{m}|}
\Big),&\text{if}\:\:\rho+k|\boldsymbol{m}|>0\:\:
\text{and}\:\:|\boldsymbol{m}|> k\rho,
\end{cases}
\end{equation*}
which is the projection onto the cone $\menge{(\rho,\boldsymbol{m})\in\RR\times\RR^d}{|\boldsymbol{m}|\le k\rho}$ (see, e.g., \cite{projcone}).

\item Set $F=\iota_{\RP}$
and consider the function $\phi_{\alpha,k}$ defined in \eqref{e:defphiak}.
Then, for every $\xi\in\RR$, we have
$$\lim_{k\to+\infty}\phi_{\alpha,k}(\xi)=
\lim_{k\to+\infty}\frac{\xi^2/\alpha}{1+
\sqrt{1-\big(\frac{\xi}{k}\big)^2}}=\frac{\xi^2}{2\alpha},$$
and we conclude from \eqref{e:aux44} that
$$\widehat{\Phi}_{c}(\rho,\boldsymbol{m}):=\lim_{k\to+\infty}\Phi^{\alpha,k}_c(\rho,\boldsymbol{m})
=\begin{cases}
\dfrac{|\boldsymbol{m}|^2}{2\alpha\rho},&\text{if}\:\:\rho>0;\\[3mm]
0,&\text{if}\:\:\rho=0\:\:\text{and}\:\: \boldsymbol{m}=0;\\
+\infty,&\text{otherwise}.
\end{cases}$$
Moreover, by taking $k\to+\infty$, we deduce from \eqref{e:proxF0} and \eqref{e:R0} that
\begin{equation*}
\label{e:proxF0kinf}
\prox_{\gamma\widehat{\Phi}_c}\colon 
(\rho,\boldsymbol{m})\mapsto
\begin{cases}
(0,0),&\text{if}\:\:\rho+\frac{\alpha|\boldsymbol{m}|^{2}}{2\gamma}\le 0;\\
\Big({\theta},\left(\frac{\theta}{\theta+\frac{\gamma}{\alpha}}\right)\boldsymbol{m}
\Big),&\text{if}\:\:\rho+\frac{\alpha|\boldsymbol{m}|^{2}}{2\gamma}>0,
\end{cases}
\end{equation*}
where $\theta\ge\max\{0,\rho\}$ is the unique solution to 
$$\left(t+\frac{\gamma}{\alpha}\right)^2(t-\rho)=\frac{\gamma|\boldsymbol{m}|^2}{2\alpha}.$$
This is a particular case of the following proposition.
\end{enumerate}
\end{remark}

\begin{prop}
\label{p:proxPhip}
Let $F\colon\RR\to\left]-\infty,+\infty\right]$ be a lower semicontinuous convex function which is differentiable in $\inte\dom F=\RPP$ and such that $L_0:=\lim_{\xi\downarrow 0}F'(\xi)\in\RR\cup\{-\infty\}$.  Moreover, let 
$\Psi\colon\RR\times\RR^d\to\left]-\infty,+\infty\right]\colon (\rho,\boldsymbol{m})\mapsto 
F(\rho)$, set $p>1$, set $q=p/(p-1)$, and set
\begin{equation*}
\Phi_p\colon (\rho,\boldsymbol{m})\mapsto \begin{cases}
\dfrac{|\boldsymbol{m}|^q}{q\rho^{q-1}},&\text{if}\:\:\rho>0;\\[3mm]
0,&\text{if}\:\:\rho=0\:\:\text{and}\:\: \boldsymbol{m}=0;\\
+\infty,&\text{otherwise}.
\end{cases}
\end{equation*}
Then, for every $\gamma>0$, we have
\begin{equation}
\label{e:calcproxPhip}
\prox_{\gamma(\Phi_p+\Psi)}\colon 
(\rho,\boldsymbol{m})\mapsto
\begin{cases}
(0,0),&\text{if}\:\:\rho+\dfrac{|\boldsymbol{m}|^p}{p\gamma^{p-1}}
\le\gamma L_0;\\[3mm]
({\theta},{\boldsymbol{v}}),&\text{if}\:\:
\rho+\dfrac{|\boldsymbol{m}|^p}{p\gamma^{p-1}}
> \gamma L_0,
\end{cases}
\end{equation}
where ${\theta}\ge\prox_{\gamma F}\rho\ge 0$ is the 
unique solution to $R^p_{\rho,\boldsymbol{m}}(t)=0$ with
\begin{equation}
\label{e:nonlineareq1Fp}
R^p_{\rho,\boldsymbol{m}}\colon
t\mapsto \Big(\gamma(t+\gamma 
F'(t)-\rho)^{1/p}+t(p/\gamma)^{1-2/p}(t+\gamma 
F'(t)-\rho)^{1-1/p}\Big)^p-\frac{\gamma|\boldsymbol{m}|^p}{p}
\end{equation} 
and 
\begin{equation}
\label{e:vhatp}
{\boldsymbol{v}}=
\left(\frac{{\theta}}{{\theta}+\gamma^{2/q}p^{1-2/q}({\theta}+\gamma
 F'({\theta})-\rho)^{1-2/q}}\right)\boldsymbol{m}.
\end{equation}
\end{prop}
\begin{proof}
Set $\phi\colon \xi\mapsto |\xi|^q/q$, which is proper, convex, even, 
differentiable in $\dom\phi=\RR$, $\phi(0)=0$, $\phi'\colon \xi\mapsto 
|\xi|^{q-2}\xi$, and we have
$\lim_{\xi\to\pm\infty}\phi'(\xi)=\pm\infty$.
Therefore, noting that 
\begin{equation*}
\label{e:aux44p}
\Phi_p\colon(\rho,\boldsymbol{m})\mapsto
\begin{cases}
\rho\,\phi\left(\dfrac{|\boldsymbol{m}|}{\rho}\right),\quad&\text{if}\:\:\rho>0;\\
0,&\text{if}\;\;\rho=0\:\:\text{and}\:\: \boldsymbol{m}=0;\\
+\infty,&\text{ otherwise}
\end{cases}
\end{equation*}
and $\phi^*\colon s\mapsto |s|^p/p$,
Proposition~\ref{p:proxgen} 
yields \eqref{e:calcproxPhip},
where ${\theta}\ge\prox_{\gamma F}\rho\ge0$ is the 
unique solution to $R^{\phi}_{\rho,\boldsymbol{m}}(t)=0$, i.e., it satisfies
\begin{equation}
\label{e:sqrteqp}
{\theta}+\gamma 
F'({\theta})-\rho=\frac{\chi^p}{p\gamma^{p-1}},
\end{equation}
where
\begin{align}
\label{e:nu2p}
\chi=\gamma\prox_{\frac{{\theta}}{\gamma}\phi^*}\bigg(\frac{|\boldsymbol{m}|}{\gamma}\bigg)
\quad &\Leftrightarrow\quad
\frac{|\boldsymbol{m}|-\chi}{{\theta}}=(\phi^*)'\left(\frac{\chi}{\gamma}\right)=\left(\frac{\chi}{\gamma}\right)^{p-1}\nonumber\\
&\Leftrightarrow\quad
\chi=\frac{|\boldsymbol{m}|}{1+\frac{{\theta}}{\gamma}\left(\frac{\chi}{\gamma}\right)^{p-2}}
\end{align}
in view of \eqref{e:proxpos} in Remark~\ref{r:1}.
Now, since \eqref{e:sqrteqp} is equivalent to
\begin{equation*}
 \chi=\gamma^{1/q}p^{1/p}({\theta}+\gamma 
 F'({\theta})-\rho)^{1/p},
\end{equation*}
we obtain from \eqref{e:nu2p} that 
\begin{equation}
\label{e:vfinp}
\chi=\frac{|\boldsymbol{m}|}{1+{\theta}\gamma^{-2/q}p^{1-2/p}({\theta}+\gamma
 F'({\theta})-\rho)^{1-2/p}}.
\end{equation}
Using \eqref{e:vfinp} in \eqref{e:sqrteqp} we conclude from 
\eqref{e:nonlineareq1Fp} that $R^\phi_{\rho,\boldsymbol{m}}({\theta})=R^p_{\rho,\boldsymbol{m}}({\theta})=0$. 
Finally, \eqref{e:vhatp} follows from \eqref{e:vhat}, \eqref{e:nu2p}, 
and \eqref{e:vfinp} noting that 
$2/p-1=1-2/q$.
\end{proof}

In view of Remark~\ref{r:1}, an upper bound for ${\theta}$ 
in the last case is
$$\overline{\theta}=\prox_{\gamma 
F}\Big(\rho+\frac{|\boldsymbol{m}|^p}{p\gamma^{p-1}}\Big).$$
Note that the result in Proposition~\ref{p:proxPhip} is essentially done in \cite{BAKS}.

\begin{remark}
Suppose that $F=\iota_{\RP}$. Then, $F$ is differentiable in $\RPP$, for every $\xi>0$, $F'(\xi)=0$, $\lim_{\xi\downarrow 0}F'(\xi)=0$,
$\prox_{\gamma F}\colon \rho\mapsto\max\{0,\rho\}$, 
$\Phi_p+F=\Phi_p$, and we deduce
\begin{equation*}
\prox_{\gamma\Phi_p}\colon 
(\rho,\boldsymbol{m})\mapsto
\begin{cases}
(0,0),&\text{if}\:\:\rho+\dfrac{|\boldsymbol{m}|^p}{p\gamma^{p-1}}\le0;\\[3mm]
\bigg({\theta},\bigg(\dfrac{{\theta}}{{\theta}+\gamma^{2/q}p^{1-2/q}({\theta}-\rho)^{1-2/q}}
\bigg)\boldsymbol{m}\bigg),&\text{if}\:\:\rho+\dfrac{|\boldsymbol{m}|^p}{p\gamma^{p-1}}> 0,
\end{cases}
\end{equation*}
where ${\theta}\ge\max\{0,\rho\}$ is the unique solution to 
$R^p_{\rho,\boldsymbol{m}}(t)=0$ with
\begin{equation*}
R^p_{\rho,\boldsymbol{m}}\colon
t\mapsto 
\Big(\gamma(t-\rho)^{1/p}+t(p/\gamma)^{1-2/p}(t-\rho)^{1/q}\Big)^p-\frac{\gamma|\boldsymbol{m}|^p}{p}.
\end{equation*}
This particular instance can also be derived from \cite[Theorem~3.1]{CoJMAA} and \cite[Proposition~2.3]{CoEJS}.
Note that, when $p=q=2$, we recover \cite[Proposition~1]{PPO}.
\end{remark}

\bigskip

\section{Numerical Results}

\subsection{$p$-Laplacian}
Let us first focus on the family of doubly nonlinear diffusion problems of the form
\begin{equation}\label{dnl}
\left\{
\begin{array}{rl}
 \displaystyle \frac{\partial\rho}{\partial t} = \nabla\cdot \left[\left|\nabla \rho^m \right|^{p-2}\nabla \rho^m \right]\,,&
 x\in\Omega\subset\RR^d, t>0\vspace{.3cm}\\
\rho(t=0)=\rho_0 \,,&x\in\Omega \subset\RR^d
\end{array}
\right.
\end{equation}
with no-flux boundary conditions when posed on a bounded domain $\Omega$, $1< p< \infty$, and $m\geq m_c:=\max\left(\frac{d-p}{d(p-1)},0\right)$. This family corresponds to the cost $c(x)=|x|^q/q$, with $q$ the conjugate exponent of $p$, and the internal energy
\begin{equation*}\label{energydensity}
U(s)= \left\{
\begin{array}{lcl}
\displaystyle \frac{1}{p-1}s\ln s,&\mbox{if}&\displaystyle  m=\frac{1}{p-1};\\ \\
\displaystyle \frac{ms^\gamma}{\gamma(\gamma -1)},
\;\gamma=m+\frac{p-2}{p-1}, &\mbox{if}&\displaystyle  m\neq
\frac{1}{p-1}.
\end{array}
\right. 
\end{equation*}
Particular cases include the porous medium and the $p$-Laplacian equations; see \cite{Vazquez,Vaz2} and the references therein. On bounded domains with Neumann boundary conditions these equations equilibrate to the average of the initial data, while on the whole space they evolve towards self-similar Barenblatt-type profiles in the mass-conservation range $m> m_c$ \cite{Agu2,ABC,Vaz2}.

The behaviour at the boundary of the support, as well as the decay at infinity on the whole space, depends on both $m$ and $p$. The finite speed of propagation regime corresponds to $m(p-1)>1$, yielding compactly supported solutions; the fast-diffusion regime corresponds to $m(p-1)<1$ with $m>m_c$, yielding heavy tails \cite{Agu2,ABC,Vaz2}. The Barenblatt self-similar solution reads
\begin{equation}\label{barenblattsol}
\rho_B(t,x)=\frac{1}{t^{d/\delta_p}}u_B\left(\frac{x}{t^{1/\delta_p}}\right),
\end{equation}
with $\delta_p:=d(p-1)(m-m_c)>0$ and
\begin{equation}\label{barenblattprofile}
u_B(y)= \left\{
\begin{array}{lcl}
\displaystyle \frac{1}{\sigma}\,\mbox{exp}\left(-\frac{p-1}{q}\frac{|y|^q}{\delta_p^{1/(p-1)}}\right), &\mbox{if}& \displaystyle m=\frac{1}{p-1};\\
\displaystyle\left(D_*-\frac{m(p-1)-1}{mp}\frac{|y|^q}{\delta_p^{1/(p-1)}}\right)_+^{\frac{p-1}{m(p-1)-1}},
&\mbox{if}&\displaystyle m\neq \frac{1}{p-1},
\end{array}
\right.
\end{equation}
where $\sigma$ and $D_*$ are determined by mass conservation,
$$
\|u_B\|_{L^1(\RR^d)}= \|\rho_B(t)\|_{L^1(\RR^d)}=\|\rho_0\|_{L^1(\RR^d)}=1.
$$
Convergence of the JKO scheme for this family has been established in \cite{Agu,AGS,CS}.

We apply the numerical scheme \eqref{eq:Yan_step_1}--\eqref{eq:Yan_step_3} to \eqref{dnl} in one dimension for three combinations of $m$ and $p$, with results reported in Figures~\ref{fig:pLap-1}--\ref{fig:pLap-3}. In all three cases the reference solution is given by \eqref{barenblattsol}--\eqref{barenblattprofile}, with $\sigma \approx 2.1495$ for $m=0.5$, $p=3$; $D_* \approx 0.6646932161$ for $m=1$, $p=3$; and $D_* \approx 1.12636223$ for $m=0.25$, $p=3$. Initial data are set by \eqref{barenblattsol}--\eqref{barenblattprofile} at $t=0.01$ or $t=0.001$, with matching time step as indicated in each caption. Each figure shows the evolution of the numerical solution overlaid with the analytical profile, alongside the number of Chambolle--Pock iterations per JKO step. When the solution is smooth the iteration count decreases over time as the solution approaches a steady state; when the solution has non-smooth transitions at the boundary of the support, the iteration count fluctuates.

\begin{figure}[h!]
    \centering
    \includegraphics[width = 0.49\textwidth]{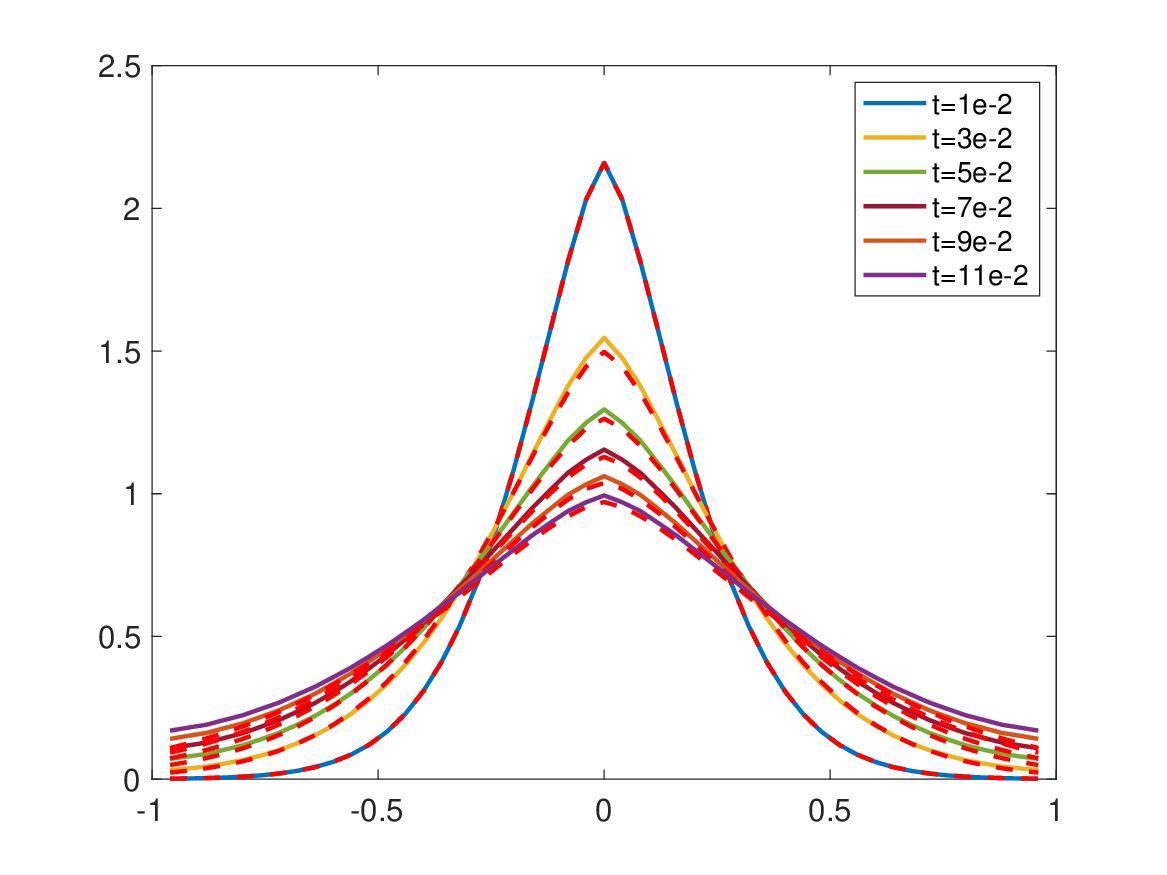}
    \includegraphics[width = 0.49\textwidth]{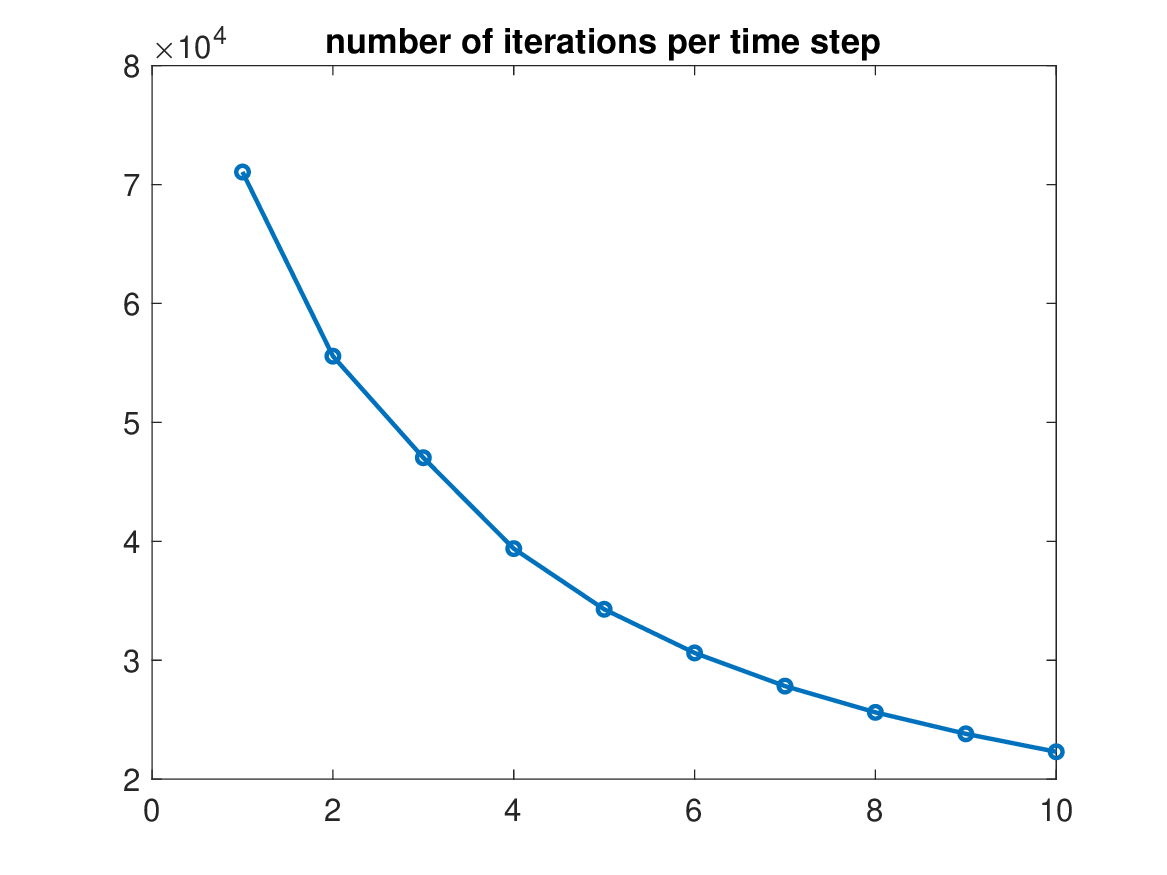}
    \caption{$p$-Laplacian case. Here $m = 0.5$, $p = 3$, so $m = \frac{1}{p-1}$. Numerical parameters: $\Delta x = 0.04$, $\Delta t = 0.01$.}
    \label{fig:pLap-1}
\end{figure}

The primal-dual framework developed in the previous sections admits two variants, which we refer to as the \emph{separate} and \emph{joint} proximal operator formulations. In the separate formulation, the term $\nabla \Psi$ associated with the internal energy is handled as an explicit gradient step within the iterations \eqref{eq:Yan_step_1}--\eqref{eq:Yan_step_3}, and the proximal operator $\prox_{\gamma \Phi_c}$ is computed via Proposition~\ref{p:proxgen}. In the joint formulation, $\Psi$ is absorbed into a single proximal operator $\prox_{\gamma(\Phi_c + \Psi)}$, whose closed-form characterisation is provided by Propositions~\ref{p:proxPhip} and~\ref{p:proxPhic} for the $p$-Laplacian and relativistic heat perspective functions, respectively. Figure~\ref{fig:pLap-2} illustrates that both variants produce essentially indistinguishable results: the Chambolle--Pock iteration counts coincide and the discrepancies between the computed densities are minimal.

\begin{figure}[h!]
    \centering
    \includegraphics[width =\linewidth]{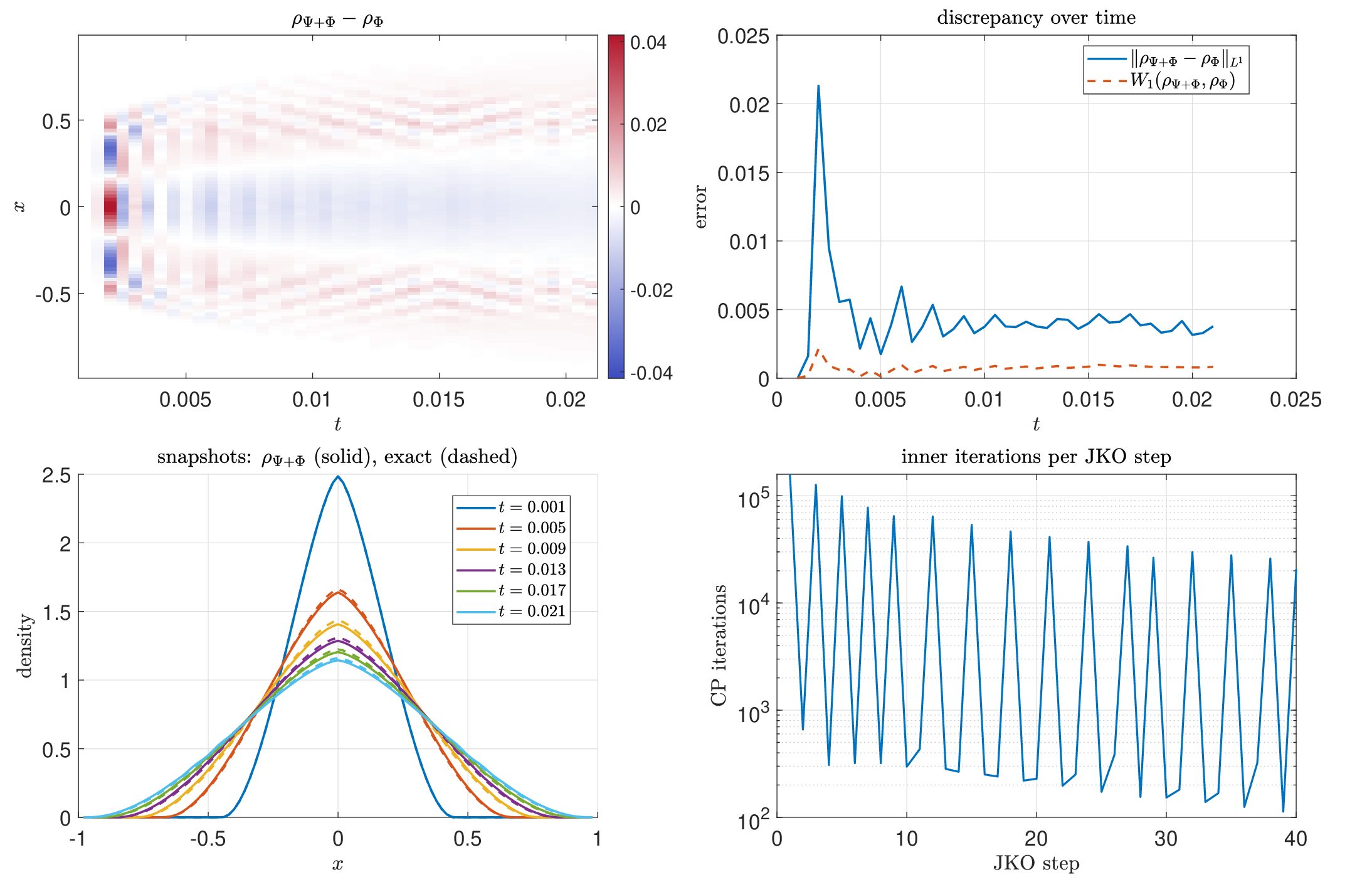}
    \caption{$p$-Laplacian case. Here $m = 1$, $p = 3$, so $m(p-1)>1$ (finite speed of propagation regime). Numerical parameters: $\Delta x = 0.02$, $\Delta t = 5\times10^{-4}$. Comparison of the joint and separate proximal operator formulations: density profiles and Chambolle--Pock iteration counts are indistinguishable between the two.}
    \label{fig:pLap-2}
\end{figure}

\begin{figure}[h!]
    \centering
    \includegraphics[width = 0.49\textwidth]{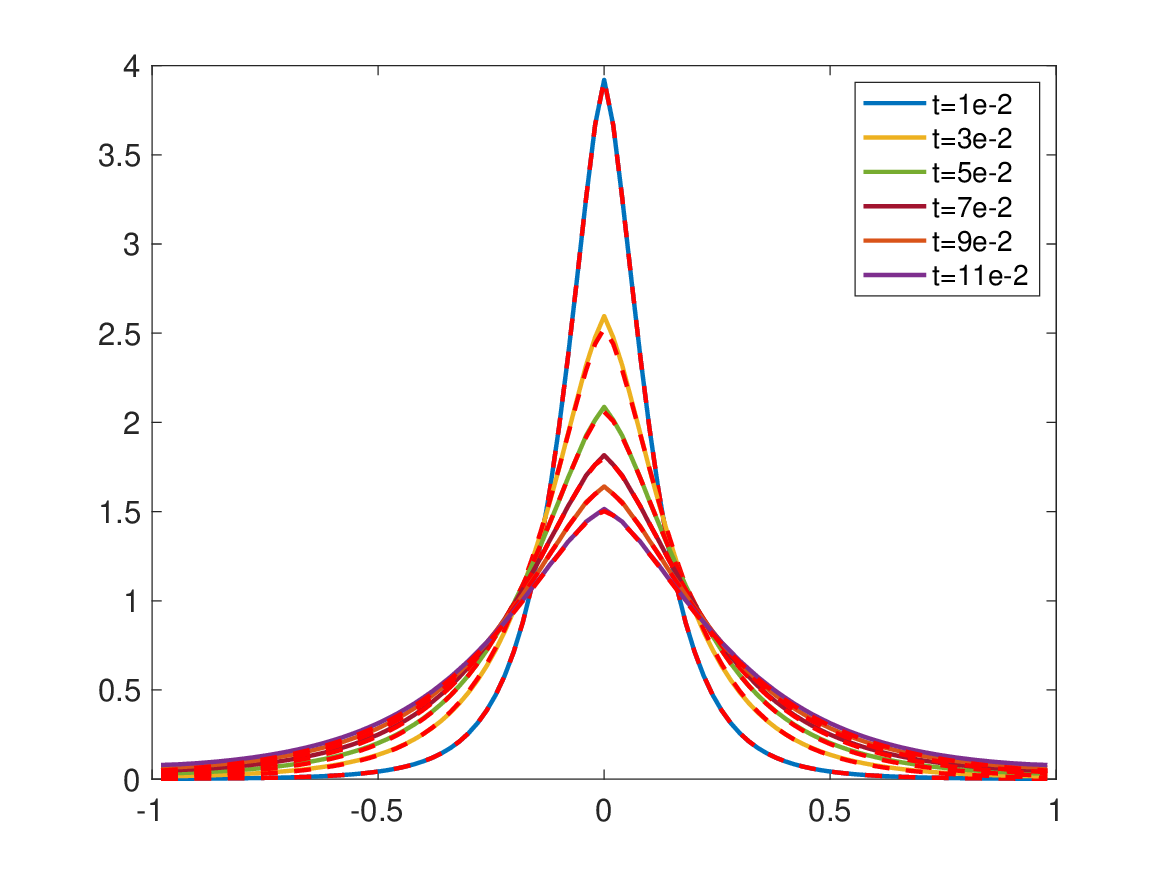}
    \includegraphics[width = 0.49\textwidth]{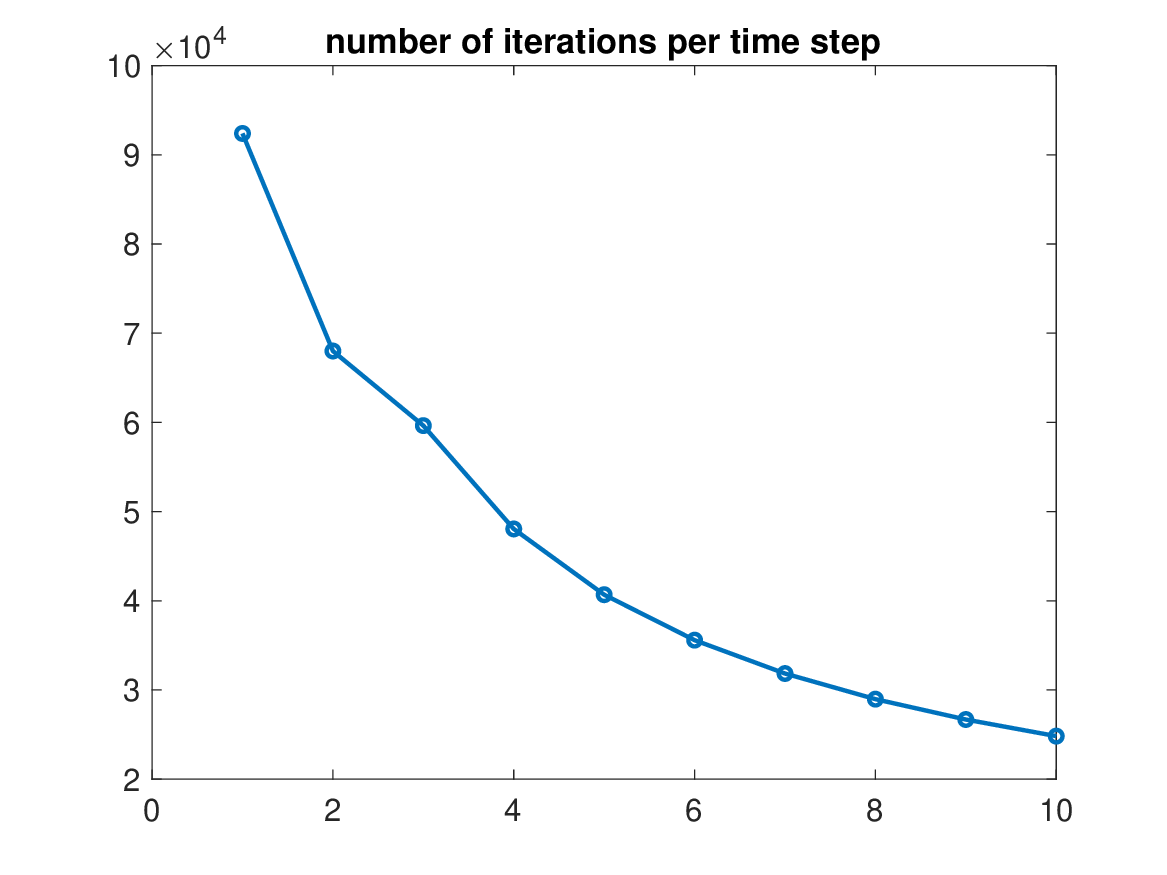}
    \caption{$p$-Laplacian case. Here $m = 0.25$, $p = 3$, so $m(p-1)<1$ (infinite speed of propagation regime). Discretization: $\Delta x = 0.02$, $\Delta t = 0.01$.}
    \label{fig:pLap-3}
\end{figure}
First-order accuracy in $\Delta t$ is verified in Figure~\ref{fig:accuracy} through relative error.
\begin{figure}[h!]
    \centering
    \includegraphics[width=0.5\linewidth]{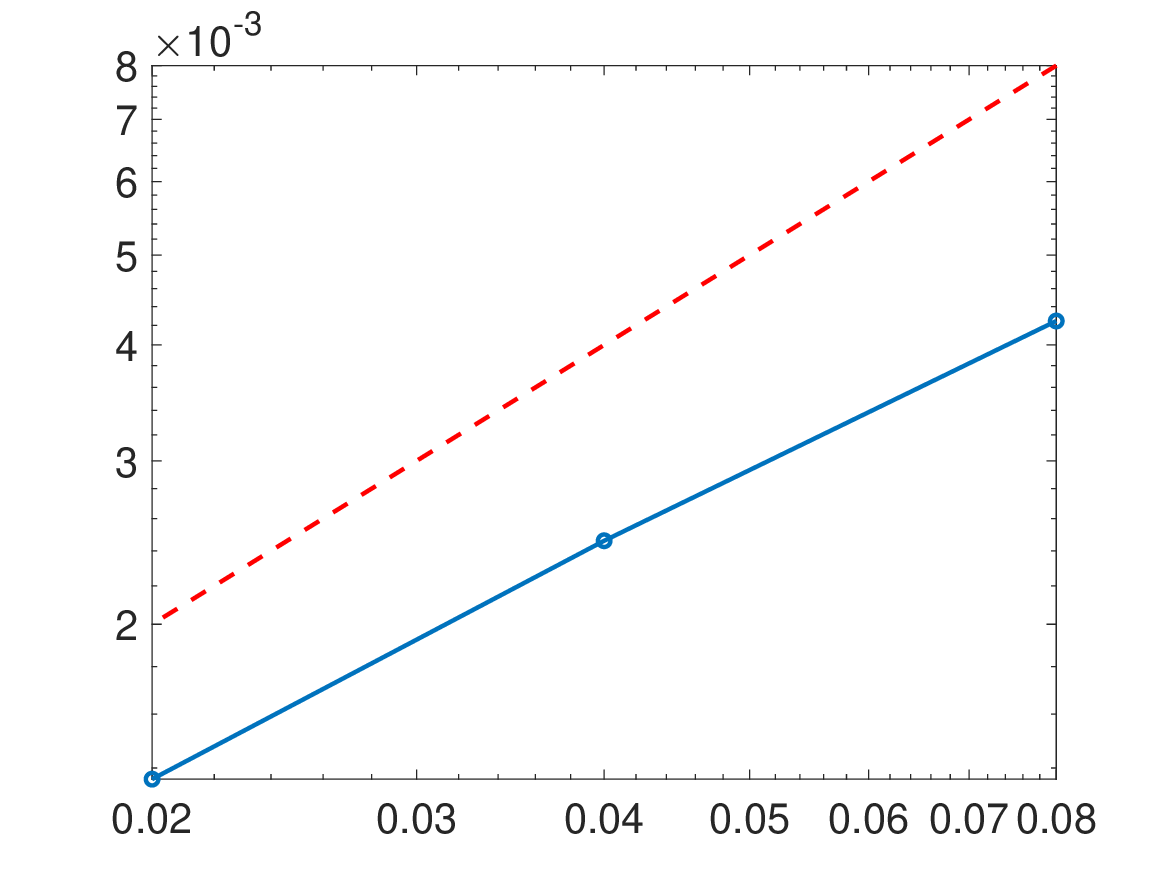}
    \caption{Error in $\Delta t$, for $m = 0.5$, $p = 3$, $\Delta x = 0.04$, and $\Delta t \in \{0.08, 0.04, 0.02, 0.01\}$.}
    \label{fig:accuracy}
\end{figure}
Figures~\ref{fig:pLap-2bump-3} and~\ref{fig:pLap-2bump-1} show the evolution of two initial bumps, taken as linear combinations of the profiles in \eqref{barenblattsol} centred at $x=0$ and $x=1$, merging into a single bump while diffusing towards the constant steady state. The corresponding entropy evolution is reported in Figure~\ref{entropy_twobump}.

\begin{figure}[h!]
    \centering
    \includegraphics[width=0.48\linewidth]{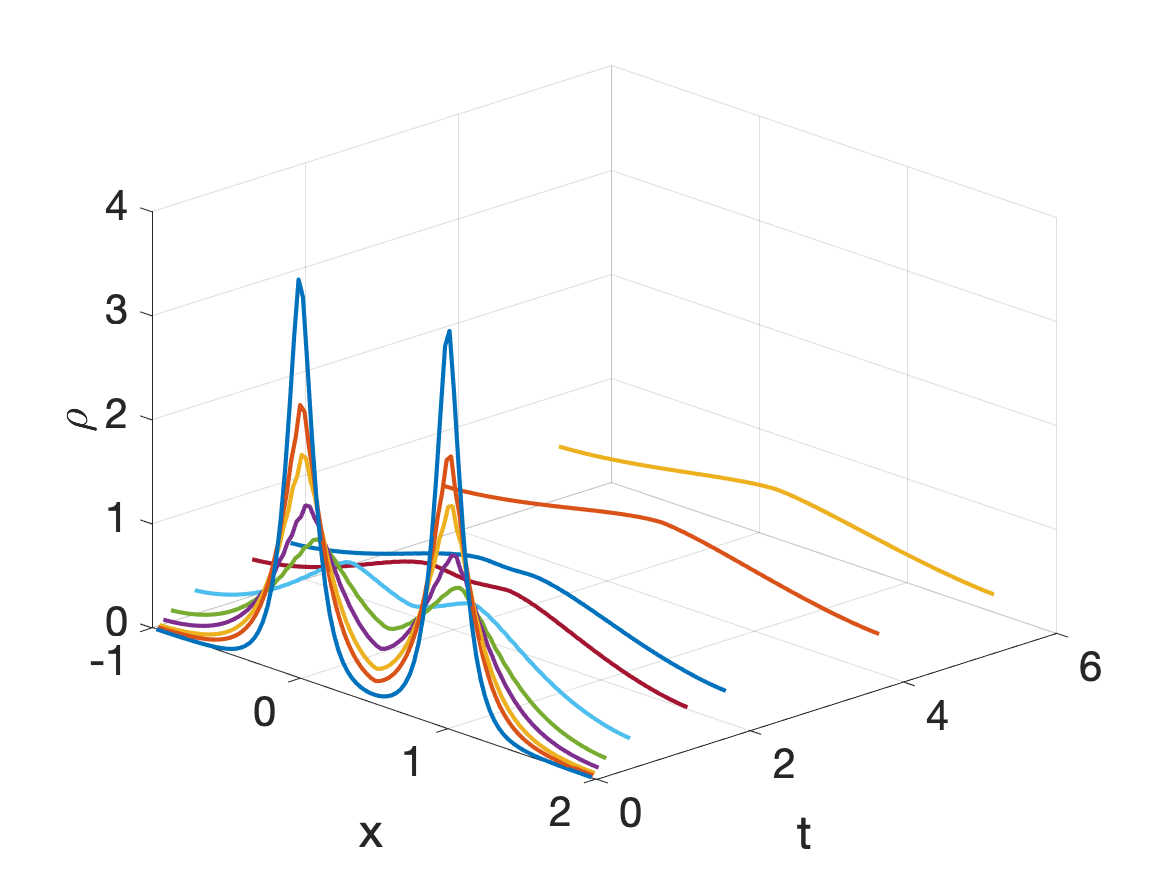}
    \includegraphics[width=0.48\linewidth]{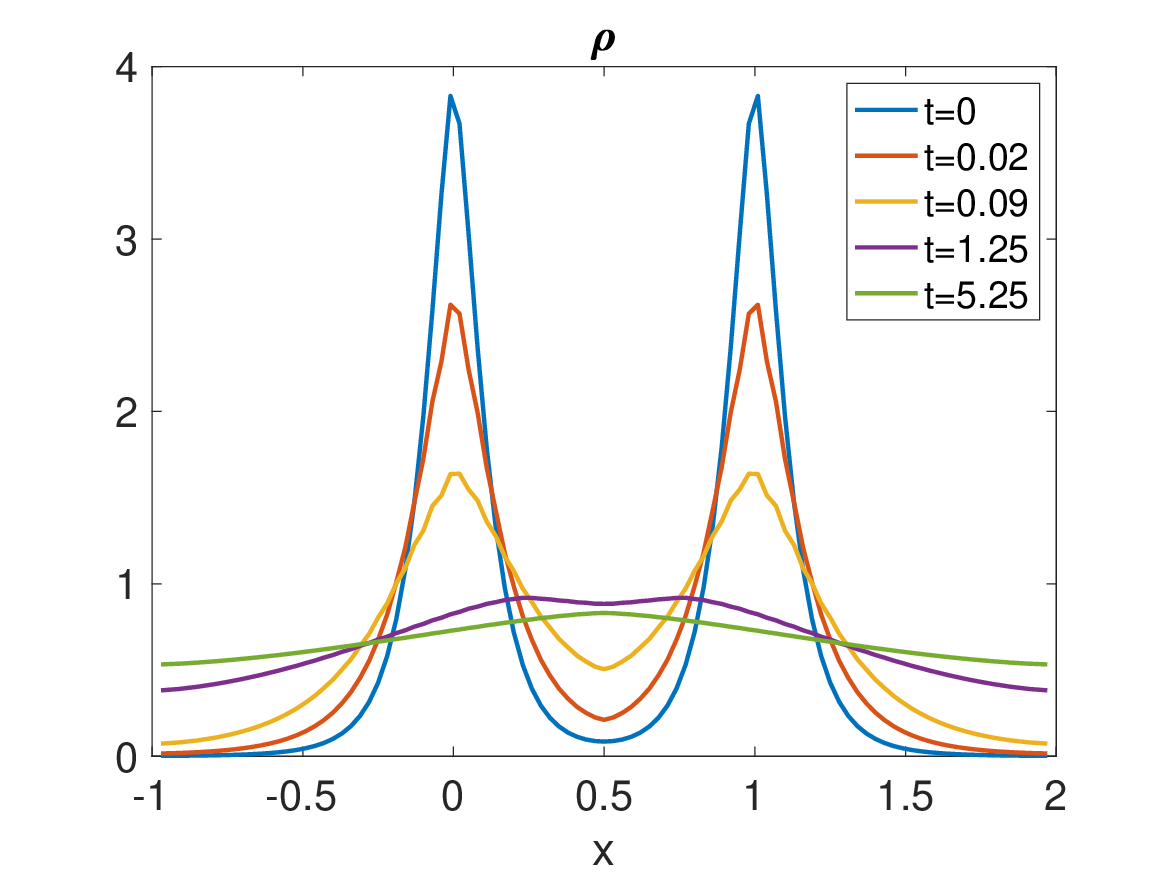}
    \caption{$p$-Laplacian case, two-bump initial data. Here $m = 0.25$, $p = 3$ (infinite speed of propagation). Discretization: $\Delta x = 0.03$; $\Delta t = 0.01$ up to $t=0.3$, $\Delta t = 0.05$ thereafter.}
    \label{fig:pLap-2bump-3}
\end{figure}

\begin{figure}[h!]
    \centering
    \includegraphics[width=0.48\linewidth]{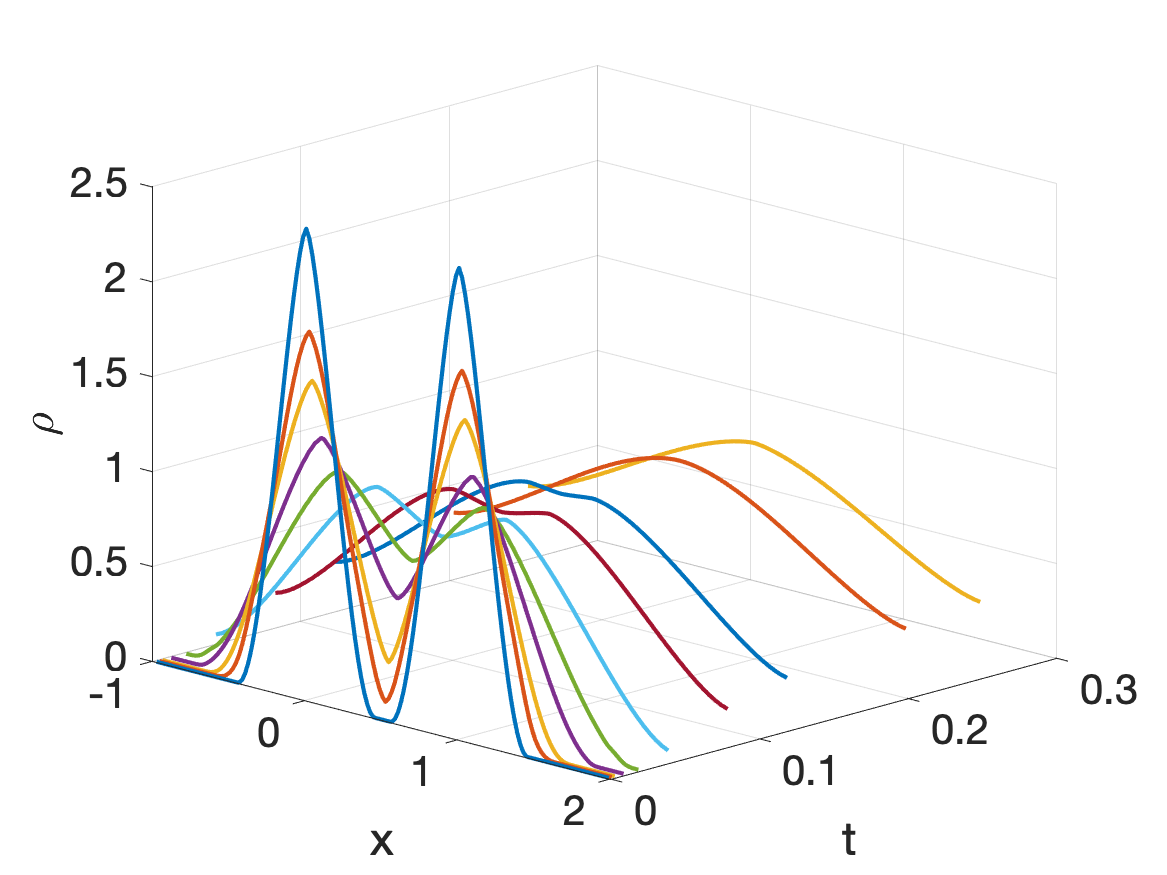}
    \includegraphics[width=0.48\linewidth]{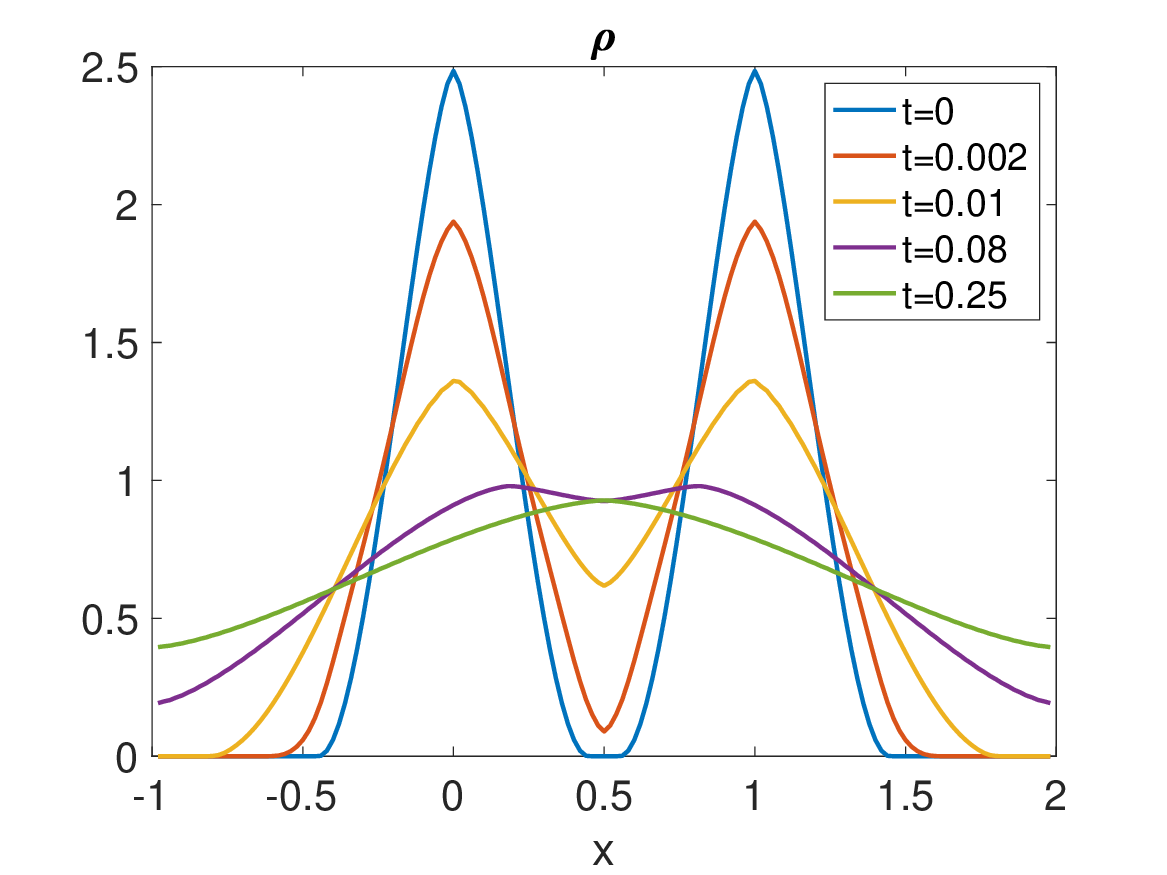}
    \caption{$p$-Laplacian case, two-bump initial data. Here $m = 1$, $p = 3$ (finite speed of propagation). Discretization: $\Delta x = 0.04$, $\Delta t = 0.001$.}
    \label{fig:pLap-2bump-1}
\end{figure}

\begin{figure}[h!]
    \centering
    \includegraphics[width=0.45\linewidth]{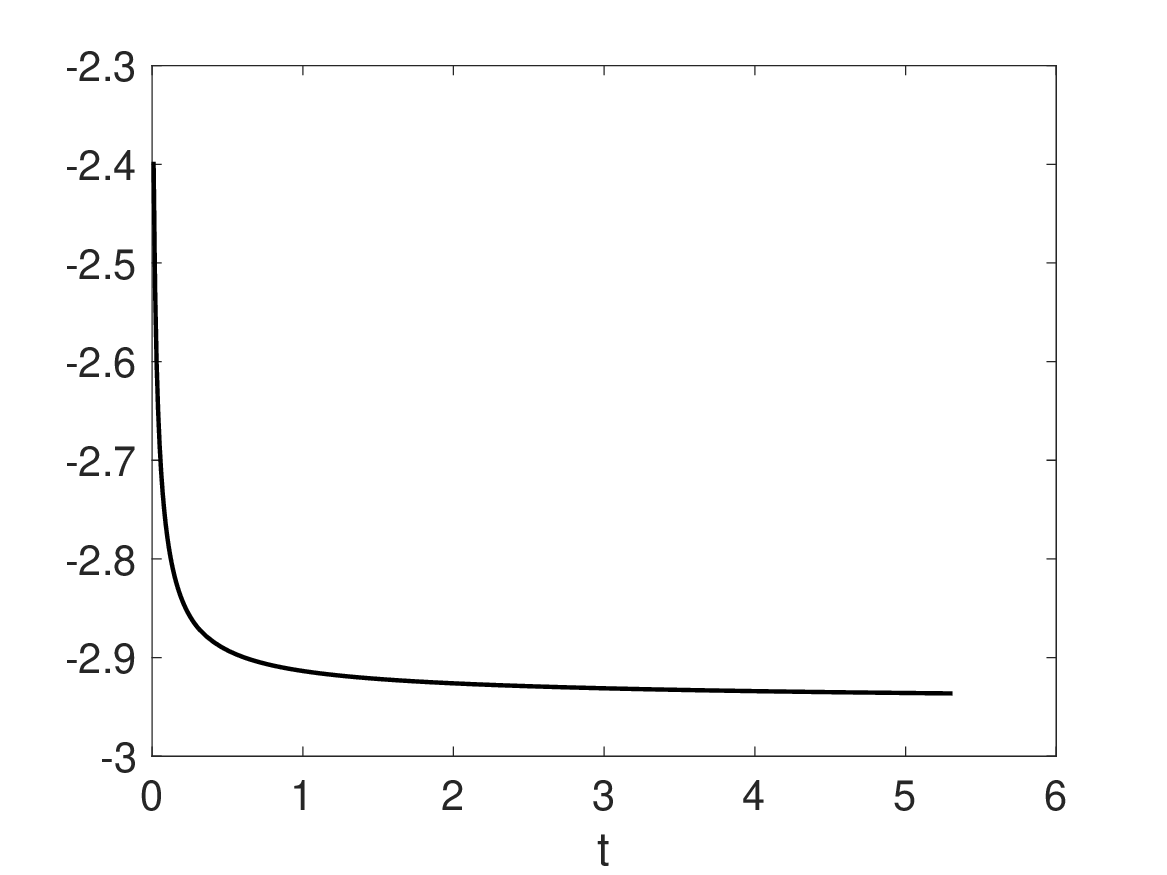}
    \includegraphics[width=0.45\linewidth]{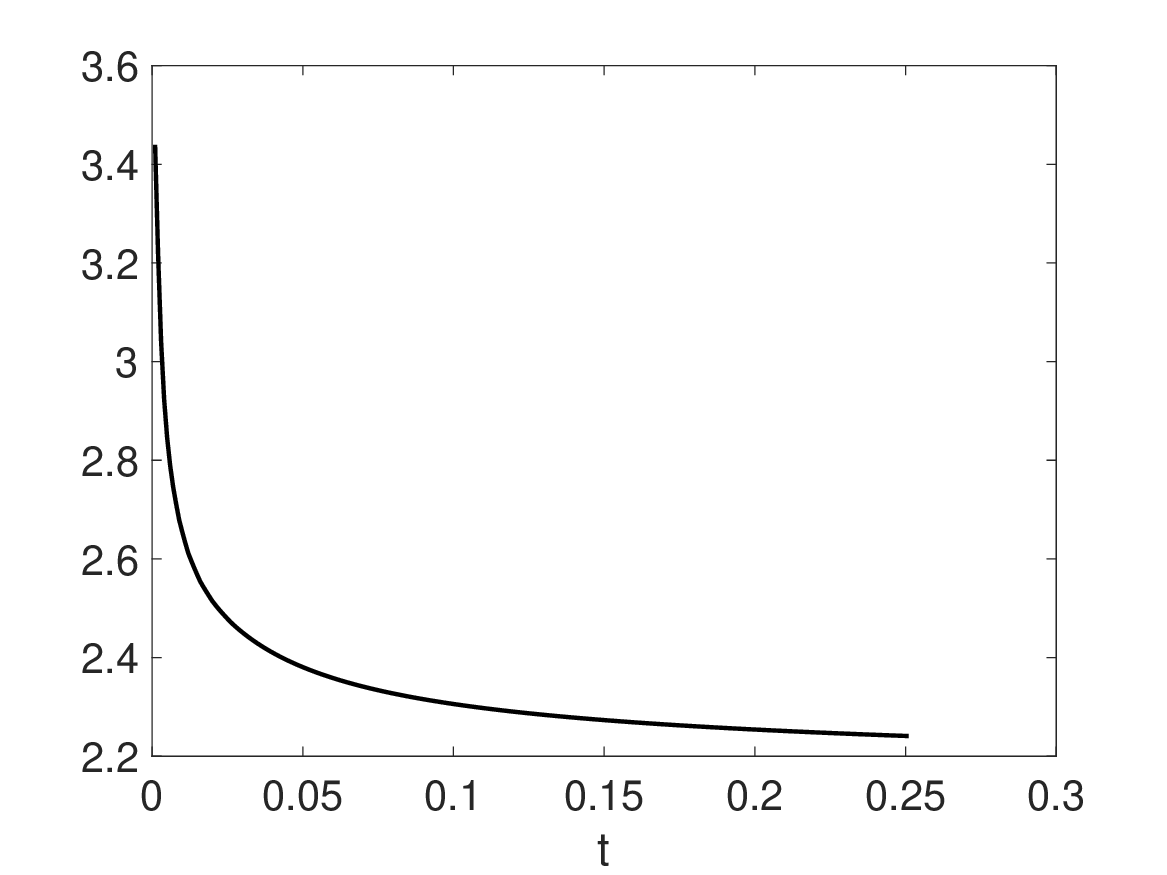}
    \caption{$p$-Laplacian case: entropy vs.\ time, two-bump initial data. Left: $m = 0.25$, $p = 3$, $\Delta x = 0.03$, $\Delta t = 0.01$ up to $t=0.3$ and $\Delta t = 0.05$ thereafter. Right: $m = 1$, $p = 3$, $\Delta x = 0.04$, $\Delta t = 0.001$.}
    \label{entropy_twobump}
\end{figure}

\subsection{Relativistic heat equation}
For the relativistic heat equation with $k,\alpha>0$ no explicit solution is available, so direct validation is not possible. We rely instead on known limiting behaviours as $k\to\infty$ and $\alpha\to\infty$, discussed below. We first consider the initial condition
\begin{equation*}
    \rho_0(x) = \frac{1}{4} \chi_{[-1,1]} + \frac{3}{2\sqrt{2}} \sqrt{\tfrac{1}{2} - |x| \chi_{[-\frac{1}{2}, \frac{1}{2}]}},
\end{equation*}
with results reported in Figure~\ref{case2:rel_H}, compared against the inverse distribution function method of \cite{CCM13}. The qualitative agreement is excellent, as is the expected energy decay. As in the $p$-Laplacian case, Figure~\ref{case2:rel_H} also serves to contrast the joint and separate proximal operator formulations: the Chambolle--Pock iteration counts coincide and the density discrepancies are minimal.

\begin{figure}
    \centering
    \includegraphics[width=\linewidth]{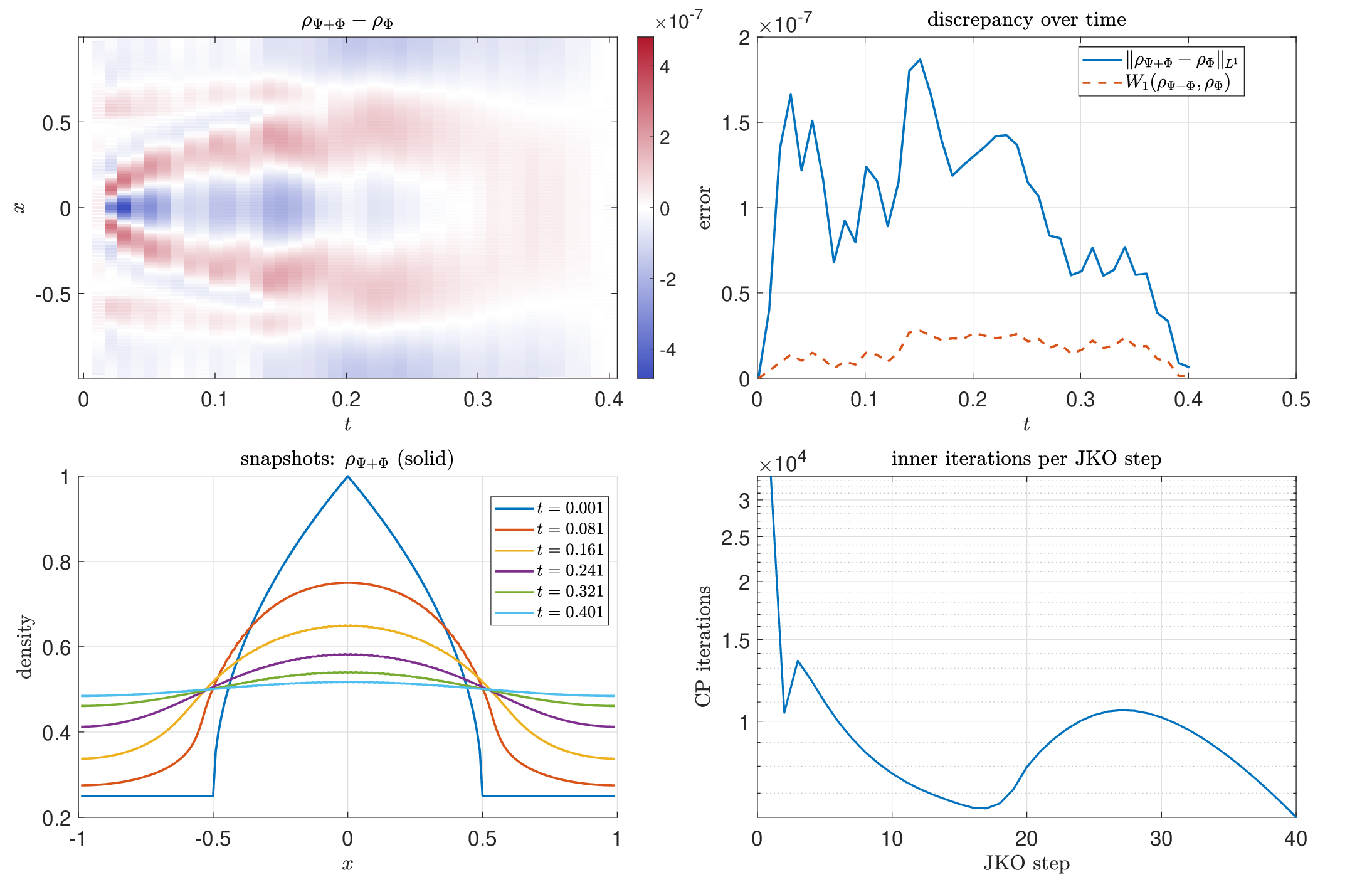}
    \caption{Relativistic heat equation with initial condition $\rho_0(x) = \frac{1}{4} \chi_{[-1,1]} + \frac{3}{2\sqrt{2}} \sqrt{\frac{1}{2} - |x| \chi_{[-\frac{1}{2}, \frac{1}{2}]}}$, $\Delta x = 0.01$, $\Delta t = 0.01$. Comparison of the joint and separate proximal operator formulations: profiles and Chambolle--Pock iteration counts coincide up to minimal discrepancies.}
    \label{case2:rel_H}
\end{figure}

For quantitative verification we consider the rescaled equation
\begin{equation}\label{s-relH}
\derpar{\rho}{t} =
\nabla\cdot\left(\rho\frac{\nabla\rho}{\sqrt{\frac{1}{\alpha^2}\rho^2+\frac{1}{k^2}|\nabla\rho|^2}}\right)=
\nabla\cdot\left(\rho\frac{\nabla\log\rho}{\sqrt{\frac{1}{\alpha^2}+\frac{1}{k^2}|\nabla\log\rho|^2}}\right),
\end{equation}
with $k$ the sound speed. This yields two limiting regimes. For $\alpha = 1$ and $k = 10^5$, equation \eqref{s-relH} approaches the heat equation; we test this with initial data
\begin{align} \label{IC:heat}
    \rho_0(x) = \frac{1}{\sqrt{4\pi t_0}}\, e^{-\frac{x^2}{4t_0}},
\end{align}
for $t_0 = 0.01$, whose exact evolution is obtained by replacing $t_0$ with $t+t_0$. For $k = 1$ and $\alpha = 10^7$, \eqref{s-relH} approaches the total variation flow; we take a smoothed characteristic function,
\begin{align} \label{IC:char}
    \rho_0(x) = 0.499\times\bigl(\tanh(30(x+0.2+t_0))-\tanh(30(x-0.2-t_0))\bigr) + 10^{-5},
\end{align}
with $t_0 = 0.001$, and expect the solution to retain a characteristic profile with boundary expanding at unit speed.

Both cases are collected in Figure~\ref{fig:heat}. The top-left panel shows heat-flow evolution and the bottom-left panel shows front expansion at unit speed. In the total variation limit, the solution diffuses as it expands, owing partly to numerical diffusion and partly to the fact that the initial data is not precisely a characteristic function. A sharper initial profile would require a finer spatial grid; even with the present resolution, some overflow is visible near the propagating front. A high-order spatial discretization would address this, and we leave it for future work. Despite these effects, the entropy decays monotonically, as shown in the second column of Figure~\ref{fig:heat}.

\begin{figure}
    \centering
    \includegraphics[width=0.48\linewidth]{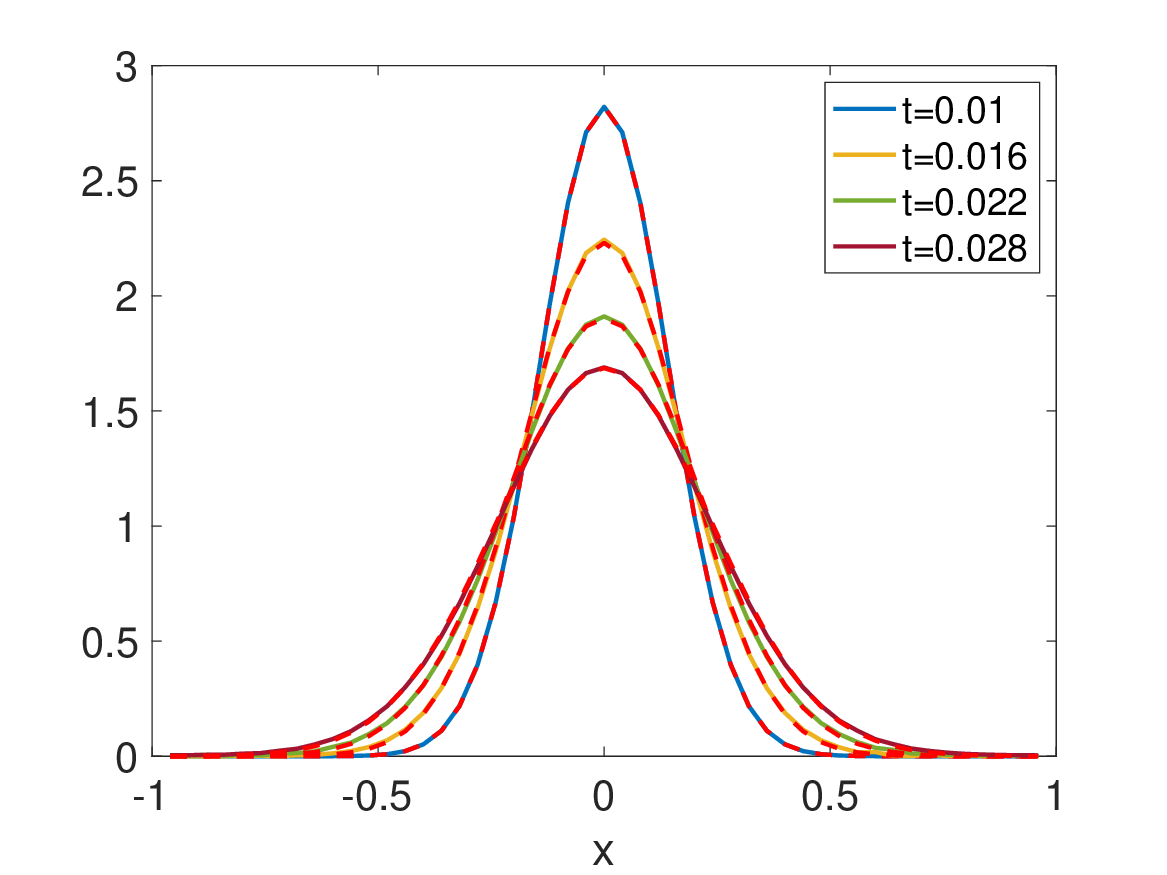}
    \includegraphics[width=0.48\linewidth]{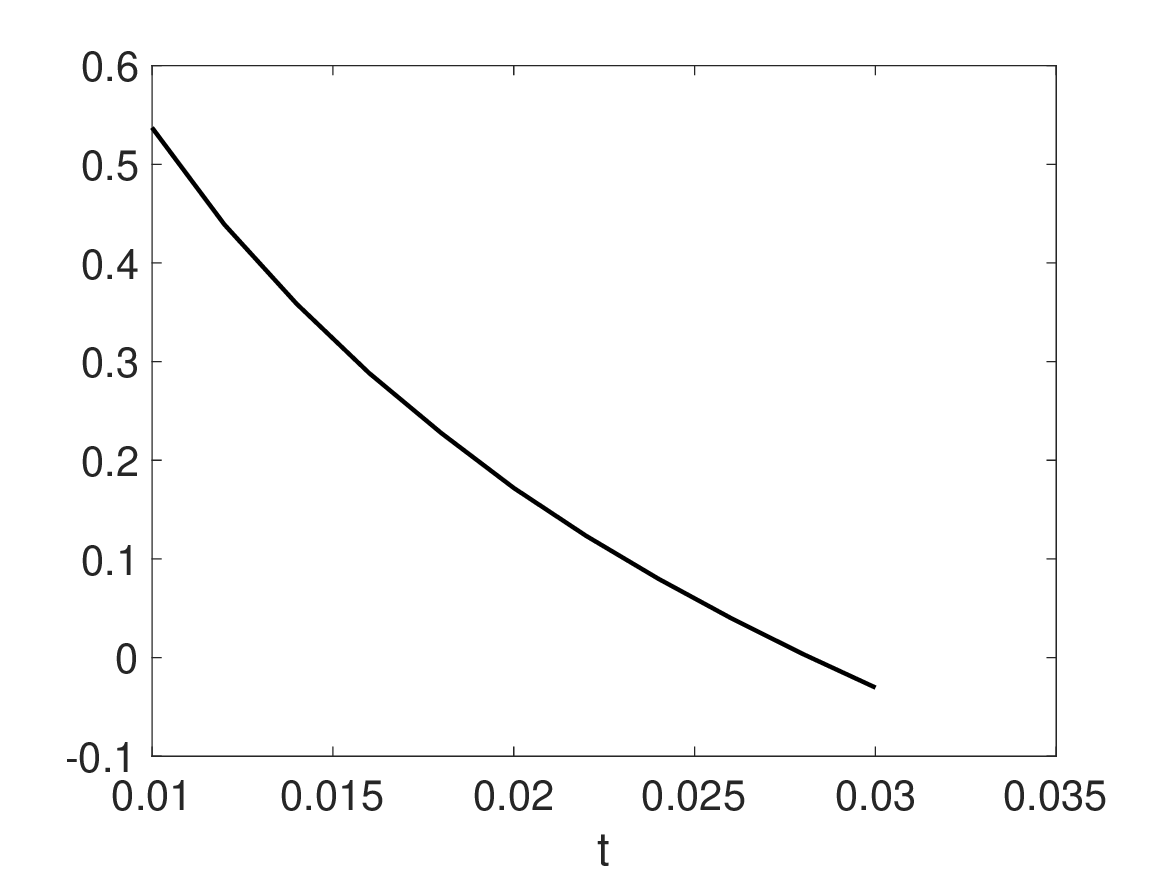}  
    \\ 
    \includegraphics[width=0.48\linewidth]{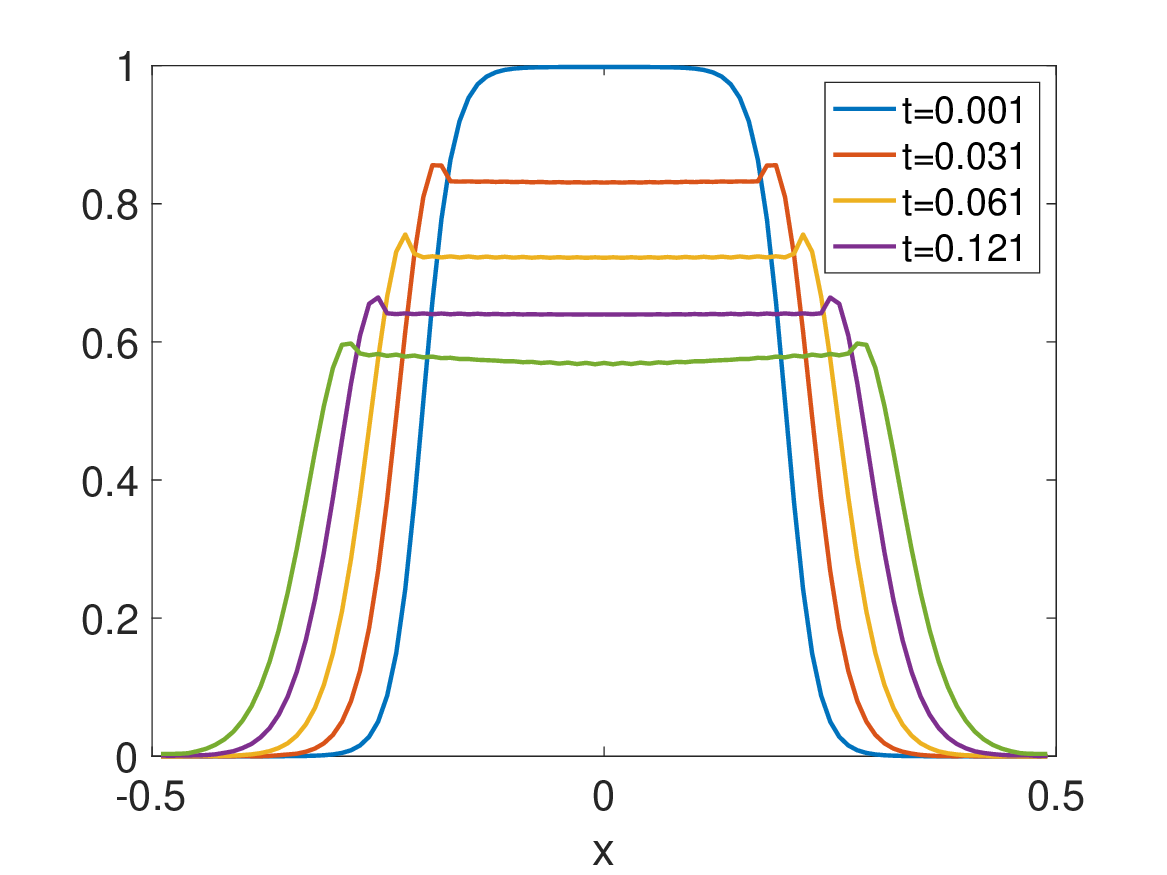}
    \includegraphics[width=0.48\linewidth]{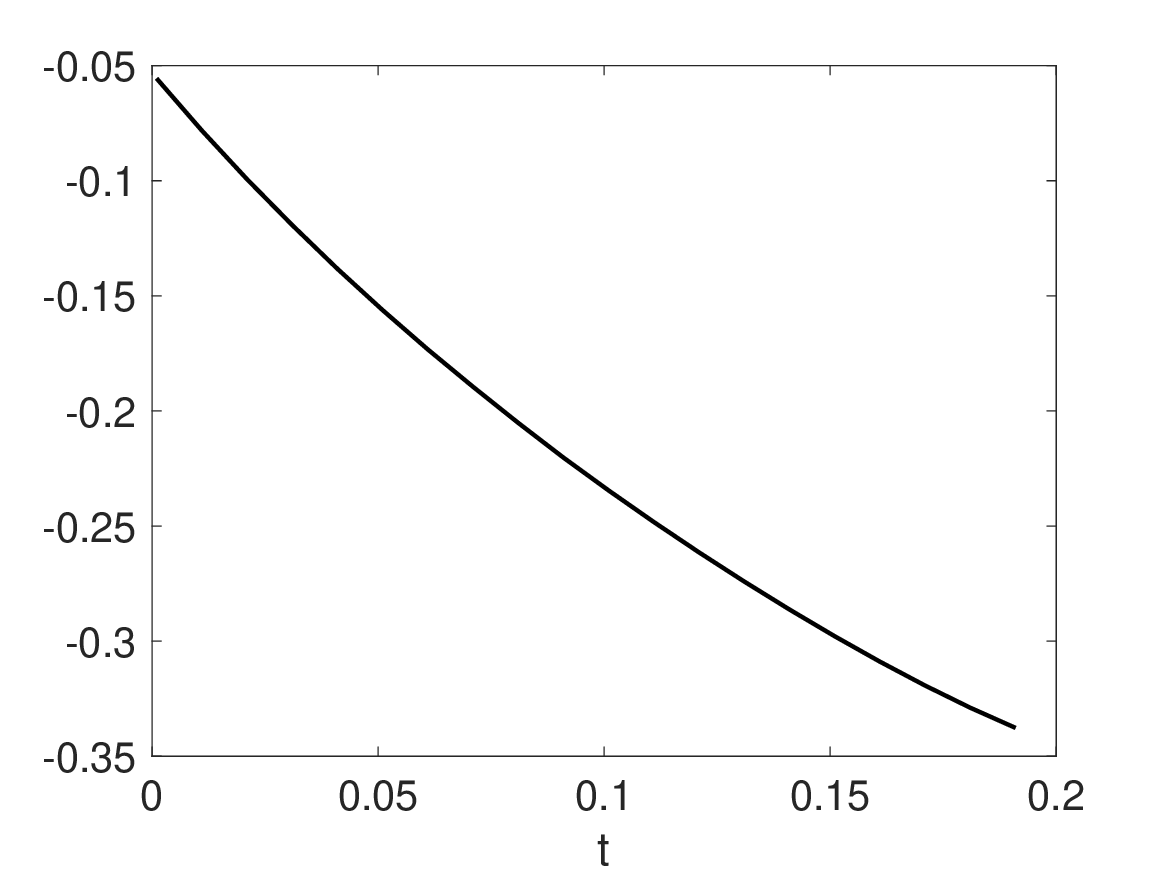}
    \caption{Relativistic heat equation. Top: $\alpha = 1$, $k = 10^5$, initial data \eqref{IC:heat}, $\Delta x = 0.04$, $\Delta t = 0.002$. Bottom: $\alpha = 10^7$, $k=1$, initial data \eqref{IC:char}, $\Delta x = 0.01$, $\Delta t = 0.01$. The second column shows the entropy vs.\ time.}
    \label{fig:heat}
\end{figure}


\section{Discussion and open problems}
In this work we have studied the use of primal-dual proximal operator methods for the numerical approximation of doubly nonlinear parabolic equations via a generalized JKO scheme.
On the numerical analysis side, the only known result is the convergence of a single variational step under smoothness conditions for the Euclidean distance ${\cal W}_2$, proven in \cite[Section 3]{CCWW}. Combining single-step variational convergence with abstract results on $\Gamma$-convergence of gradient flows, such as those in \cite{S11,Mielke}, may lead to convergence theorems for the primal-dual methods of \cite{CCWW,CWW} and of the present work. This remains a challenging open problem. Quantifying the convergence rate of primal-dual methods analytically, beyond what has been explored numerically, is also a very interesting direction, even for the simpler case of linear Fokker-Planck equations.

From a computational viewpoint, it is natural to extend our approach to the family of partial differential equations of the form
\begin{equation}\label{contgenaux}
\derpar{\rho}{t} = \nabla\cdot \left\{
\rho\,\nabla c^*\left[ \nabla \left ( U' \left ( \rho \right ) + V +
W\ast \rho \right ) \right]\right\},
\end{equation}
with general confining potential $V$ and interaction potential $W$. The presence of $W$ increases the numerical cost of solving \eqref{contgenaux} due to its convolution structure, and finding efficient primal-dual algorithms that handle this convolution efficiently would be desirable. The structural setting is similar to the one studied here: equation \eqref{contgenaux} is a nonlinear continuity equation whose velocity field is nonlinearly related to the density itself and given by $-\nabla c^*\left[ \nabla \frac{\delta {\cal F}}{\delta \rho}\right]$, with $\frac{\delta {\cal F}}{\delta \rho}=U' \left ( \rho \right ) + V + W\ast \rho$ and ${\cal F}$ the free-energy functional
\begin{equation} \label{Faux}
{\cal F}(\rho) = \int_{\R^d} U(\rho)\,\dd x + \int_{\R^d}
V(x)\,\rho(x)\,\dd x + \frac12 \int_{\R^d\times\R^d}
W(x-y)\,\rho(x)\,\rho(y)\,\dd x\,\dd y.
\end{equation}
The interaction potential is a natural modelling ingredient, particularly in mathematical biology, where chemotactic fields and cell interactions are often described through such terms. As mentioned earlier, flux-limited versions of the classical Keller-Segel model for chemotaxis have been introduced along these lines \cite{BBNS12,PVW20}, combining both effects.

Finally, the treatment of singular interaction potentials requires care. Splitting schemes for primal-dual proximal methods rely on a Lipschitz gradient assumption on the energy, as in \cite{DavisYin17}, and the proximal operators developed in Propositions~\ref{p:proxPhic} and~\ref{p:proxPhip} assume differentiability of the functional involved. Interaction energies with singular potentials generically fail to have a Lipschitz gradient and are typically non-convex. Exploring, even numerically, how far primal-dual methods can be pushed in this regime is a research avenue of considerable interest.

\section*{Acknowledgements} JAC was supported by the Advanced Grant Nonlocal-CPD (Nonlocal PDEs for Complex Particle Dynamics: Phase Transitions, Patterns and Synchronization) of the European Research Council Executive Agency (ERC) under the European Union Horizon 2020 research and innovation programme (grant agreement No. 883363) and  partially supported by the EPSRC EP/V051121/1. JAC was partially supported by the ``Maria de Maeztu'' Excellence Unit IMAG, reference CEX2020-001105-M, funded by MCIN/AEI/10.13039/501100011033/.

LBA was supported by Centro de Modelamiento Matem\'atico (CMM), FB210005, BASAL fund for centers of excellence, and by grants FONDECYT 1230257 and MATH AmSud 23-MATH-17 from the National Agency of Research and Development (ANID) of Chile.

FJS was partially supported by l'Agence
Nationale de la Recherche (ANR), project ANR-22-CE40-0010, by KAUST through the subaward
agreement ORA-2021-CRG10-4674.6, and by Minist\`ere de l'Europe et des Affaires \'etrang\`eres (MEAE), project MATH AmSud 23-MATH-17.

LW was partially supported by NSF grant DMS-1846854, DMS-2513336, and the Simons Fellowship.
\bibliographystyle{abbrv}
\bibliography{ref.bib}

\end{document}